\documentclass[reqno,10pt]{article}
\usepackage{amsmath,latexsym, amsfonts, amssymb, amsthm}
\usepackage{graphicx,color,hyperref,dsfont,tikz,caption,wrapfig,subfig}
\usepackage{makeidx}
\usepackage{pstricks}

\usepackage{pgf,pgfarrows,pgfnodes,pgfautomata,pgfheaps,pgfshade,color,pstricks,pst-plot,pst-tree,pst-eps,multido,pst-node,pst-grad,pst-eps,hyperref}

\hypersetup{
    linktoc=page,
    linkcolor=red,          
    citecolor=blue,        
    filecolor=blue,      
    urlcolor=cyan,
    colorlinks=true           
}

\numberwithin{equation}{section}

\newtheorem{theorem}{Theorem}[section]
\newtheorem{lemma}[theorem]{Lemma}
\newtheorem{proposition}[theorem]{Proposition}

\newtheorem{rem}[theorem]{Remark}

\def \P{ \mathbb P  }
\def \E{ \mathbb E  }
\def \tP{ \widetilde{\mathbb P}  }
\def \tE{\widetilde{ \mathbb E} }

\renewcommand{\triangle}{\Delta}

\newcommand{\R}{\mathbb{R}}

\newcommand{\N}{\mathbb{N}}

\newcommand{\U}{\mathbb{U}}
\newcommand{\ind}{\mathds{1}}

\renewcommand{\ge}{\geq}
\renewcommand{\le}{\leq}
\newcommand{\cA}{\mathcal{A}}
\newcommand{\cB}{\mathcal{B}}
\newcommand{\cF}{\mathcal{F}}
\newcommand{\dd}{\mathrm{d}}   
\DeclareMathOperator{\argmin}{\mathrm{argmin}}

\newcommand{\gep}{\varepsilon}       

\newcommand{\gD}{\Delta}

\newcommand{\gl}{\lambda}
\newcommand{\gL}{\Lambda}

\newcommand{\cL}{\mathcal L}
\newcommand{\tf}{\textsc{f}}

\renewcommand{\log}{\ln}
\renewcommand{\epsilon}{\varepsilon}
\newcommand{\hub}{\color{red}}
\newcommand{\cX}{\mathcal X}

\title{Semiclassical limit of Liouville Field Theory}

\date{}

\begin{document}

\maketitle
\begin{center}

{\Large  Hubert Lacoin \footnotemark[1],\ R\'emi Rhodes \footnotemark[1]\footnotemark[3],\ 
Vincent Vargas \footnotemark[2]\footnotemark[3]}

\bigskip

\footnotetext[1]{Universit{\'e} Paris-Dauphine, Ceremade, F-75016 Paris, France. Corresponding author: rhodes@ceremade.dauphine.fr, tel: 0033144054851.} \footnotetext[2]{Ecole Normale Sup\'erieure, DMA, 45 rue d'Ulm,  75005 Paris, France.} \footnotetext[3]{Partially supported by grant ANR-11-JCJC  CHAMU}

\end{center}

\begin{abstract}
Liouville Field Theory (LFT for short) is a two dimensional model of random surfaces, which is for instance involved in $2d$ string theory or in the description of the fluctuations of metrics in $2d$ Liouville quantum gravity. This is a probabilistic model that consists in weighting the classical   Free Field action with an interaction term given by the exponential of a Gaussian multiplicative chaos. The main input of our work is the study of the semiclassical limit of the theory, which is a prescribed asymptotic regime 
of LFT of interest in physics literature (see \cite{witten} and references therein). We derive exact formulas for the Laplace transform of the Liouville field in the case of flat metric on the unit disk with Dirichlet boundary conditions. As a consequence, we prove that the Liouville field  concentrates on the solution of the classical Liouville equation   with explicit negative  scalar curvature. We also characterize the leading fluctuations, which are Gaussian and massive, and establish a large deviation principle. Though considered as an ansatz in the whole physics literature, it seems that it is the first rigorous probabilistic derivation of the semiclassical limit of LFT.  On the other hand, we carry out the same analysis  when we further weight the Liouville action with heavy matter operators. This procedure appears when computing the $n$-points correlation functions of LFT. 
 \end{abstract}
 
\vspace{0.2cm}


\noindent{\bf Key words or phrases:} Liouville equation, Gaussian multiplicative chaos, semiclassical limit, large deviation principle,  Liouville field theory, singular Liouville equation.

\vspace{0.2cm}



 \makeatletter \renewcommand{\@dotsep}{10000} \makeatother

\normalsize

\section{Introduction}
One of the aims of this paper is to initiate the study of Laplace asymptotics and large deviation principles in the realm of Liouville field theory. 

To begin with, let us mention that there exists a considerable literature devoted to Laplace asymptotic expansions and large deviation principles 
for the canonical random paths: the Brownian motion in $\R^d$. To make things simple, the aim of these studies is to investigate the asymptotic behaviour as $\gamma\to 0$ of  
\begin{equation}\label{beginning}
\E[G(\gamma B)e^{-\gamma^{-2}F(\gamma B)}] 
\end{equation}
where $B$ is a Brownian motion and $F,G$ are general functionals. Schilder's  pioneering work   
\cite{Schilder} (see also \cite{pincus}) treated the full asymptotic expansion in the case of Wiener integrals. This was then extended by Freidlin and Wentzell \cite{freidlin} to It\^o diffusions.
Similar results were obtained for conditioned Brownian paths (such as the Brownian bridge)  by Davies and Truman \cite{davies1,davies2,davies3,davies4}. 
Ellis and Rosen \cite{ellis1,ellis2,ellis3} also developed further  Laplace asymptotic expansions for Gaussian functional integrals.
Then Azencott and Doss \cite{doss}  used asymptotic expansions to study the semiclassical limit of the Schr\"odinger equation 
(see also Azencott \cite{azencott1,azencott2}). 
These works initiated   a long series   (see  for instance \cite{UBA1,UBA2,UBA3,UBA4}) and it is beyond the scope of this paper to review the whole literature until nowadays.

 There is an important conceptual difference between canonical random paths and canonical random surfaces. Whereas Brownian motion and its variants are rather nicely behaved
 (H\"older continuous), the canonical two dimensional random surface, i.e. the Gaussian Free Field (GFF), is much wilder: 
 it cannot be  defined pointwise for instance and must be understood as a random distribution (like e.g.\ the local time of the one dimensional Brownian motion, or its derivative). 
 As a consequence, many nonlinear functionals defined solely on the space of continuous functions must be defined via renormalization techniques  when applied to the GFF: see the book of Simon \cite{Simon} for instance. In this paper, we consider probably the most natural  framework of weighted random surfaces:  the $2d$-Liouville Field Theory (LFT). LFT is ruled by the   Liouville action 
\begin{equation}\label{actiondefolieetbut}
S_L(\varphi)=\frac{1}{4\pi}\int_D \big[|\partial^{\hat{g}}\varphi |^2_{\hat{g}}+QR_{\hat{g}}\varphi+4\pi\mu e^{\gamma\varphi}\big]\,\lambda_{\hat{g}}(\dd x)
\end{equation}
 in the background metric $\hat{g}$ ($\partial^{\hat{g}}$, $R_{\hat{g}}$ and $\lambda_{\hat{g}}$ stand for the gradient, curvature and volume form of the metric  $\hat{g}$) with $Q=\frac{\gamma}{2}+\frac{2}{\gamma}$, $\gamma\in ]0,2]$ and $\mu>0$. This is a model describing random surfaces or metrics. Informally, the probability  to observe  a surface in $D\varphi$  is proportional to \begin{equation}\label{path}
 e^{-S_L(\varphi)}D\varphi
 \end{equation} where $D\varphi$ stands for the ``uniform measure'' on surfaces. 
 Recall that, in the physics literature,  the Liouville action enables to describe random metrics in Liouville quantum gravity in the conformal gauge as introduced by Polyakov \cite{Pol}   (studied by David \cite{cf:Da} and Distler-Kawai \cite{DistKa}, see also the seminal work of Knizhnik-Polyakov-Zamolodchikov \cite{cf:KPZ} in the light cone gauge) or in $2d$ string theory (see Klebanov's review \cite{Kle} for instance). There are many excellent reviews on this topic \cite{witten,nakayama,Pol,teschner}. The rough idea is to couple the action of a conformal matter field (say a planar model of statistical physics at its critical point so as to become conformally invariant) to the action of gravity. This gives a couple of random variables $(e^{\gamma \varphi}\hat{g},M)$, where the random metric $e^{\gamma \varphi}\hat{g}$ encodes  the structure of the space and $M$ stands for the matter field. Liouville quantum gravity in $2d$ can thus be seen as a toy model to understand in quantum gravity how the interaction with matter influences the  geometry of space-time. Working in $2d$ is simpler because the Einstein-Hilbert action becomes essentially trivial in $2d$. In the conformal gauge and up to omitting some details, the law 
 of this pair of random variables tensorizes \cite{Pol,cf:Da} and the marginal law of the metric $e^{\gamma \varphi}\hat{g}$ is given by the Liouville action \eqref{actiondefolieetbut}. The only way the metric keeps track of its interaction with the matter field $M$ is through the parameter $\gamma$, called Liouville conformal factor, which can be explicitly expressed in terms of the central charge $c$ of the matter field using the celebrated KPZ result \cite{cf:KPZ}
\begin{equation}\label{eq:central}
\gamma=\frac{\sqrt{25-c}-\sqrt{1-c}}{\sqrt{6}}.
\end{equation}
Therefore, the influence of the matter on the space is parameterized by  $\gamma$.

 To give a rigorous meaning to the definition \eqref{actiondefolieetbut} has been carried out in \cite{DKRV} in the case of the Riemann sphere. We will restrict here to a simpler situation. We will consider the flat unit disk and impose Dirichlet   boundary conditions,  in which case the path integral \eqref{path} with action \eqref{actiondefolieetbut} is also known as Hoegh-Krohn model \cite{Hoegh}.  In that case, we have to interpret the term $$\exp\Big(-\frac{1}{4\pi}\int_D |\partial^{\hat{g}}\varphi |^2_{\hat{g}}\,\lambda_{\hat{g}}(\dd x)\Big)D\varphi$$ as the law of the centered Gaussian Free Field with Dirichlet boundary condition (GFF for short, see \cite{dubedat}). Therefore,
 we replace the Brownian motion in \eqref{beginning} by a shifted GFF $\varphi$ and the nonlinear functional $F$ in \eqref{beginning} is the integrated exponential of this GFF
\begin{equation}\label{laplace}
\E[G(\gamma \varphi)e^{-\frac{4\pi \Lambda}{\gamma^2}\int_D e^{\gamma \varphi }\lambda_{\hat{g}}}].
\end{equation}
Such an exponential, also called Liouville interaction term, is nothing but a Gaussian multiplicative chaos in $2d$. Recall that the theory of Gaussian multiplicative chaos, founded in 1985 by Kahane \cite{cf:Kah}, enables to make sense of the exponential of the GFF though the exponential is not defined on the abstract functional space the GFF lives on.  Our main motivation for considering this framework is to compute the semiclassical limit of $2d$-Liouville  Field Theory.

\begin{figure}[h]
\centering
\subfloat[curvature=1]{\includegraphics[width=0.3\linewidth]{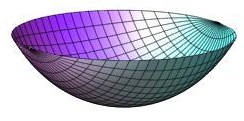}} 
\,\,\subfloat[curvature=4]{\includegraphics[width=0.3\linewidth]{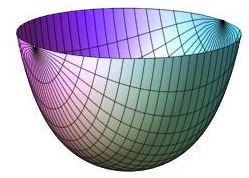}}
\caption{Two surfaces with negative curvature}
\label{bob}
\end{figure}

In a way, we will see that the geometry of space is encoded in the quantity $\mu \gamma^2$. More precisely, we study  the so-called semiclassical limit,
meaning the convergence of the field $\gamma\varphi$ when the parameter $\gamma\to 0$ while keeping fixed the quantity $\Lambda=\mu \gamma^2$ (thus $\mu\to\infty$). 
In the case where $\hat{g}$ is flat, we prove that the field $\gamma \varphi$ (resp. the measure $e^{\gamma \varphi}\,\lambda_{\hat{g}}(\dd x)$)  
converges in law towards the solution $U$ (resp. $e^{U(x)}\lambda_{\hat{g}}(\dd x)$) to the so-called classical Liouville equation
\begin{equation}\label{intro:liouville}
\triangle_{\hat{g}} U-R_{\hat{g}} =8\pi^2\Lambda  e^U.
\end{equation}
Equation \eqref{intro:liouville} appears
when looking for metrics in the conformal equivalence class of $\hat{g} $ with prescribed Ricci scalar curvature $-8\pi^2\Lambda$ (see \cite{troyanov} for instance).
Its solution has the following explicit form on the unit disk $\U$ (equipped with the flat metric, i.e. take $\hat{g}$ equal to the Euclidean metric)
 $$U(x)=2\ln \frac{1-\alpha}{1-\alpha |x|^2},\quad \text{with } \pi^2\Lambda=\frac{\alpha}{(1-\alpha)^2},$$
 with $\alpha \in ]0,1[$ when imposing Dirichlet boundary condition on $\partial \U$. The quantity $-8\pi^2\Lambda$   describes the expected curvature of the random metric  $e^{\gamma \varphi} \hat{g} $ (see Figure \ref{bob}) for small $\gamma$. The fact that this quantity is negative reflects the hyperbolicity  of the geometry of the space.    

We also characterize the leading order fluctuations around this hyperbolic geometry.  They are Gaussian and massive in the sense that the rescaled field $ \varphi-\gamma^{-1}U$ converges towards a massive Gaussian free field in the metric $e^{U(x)}\hat{g}$. The  mass of this free field is $8\pi^2\Lambda$ and thus exactly corresponds to minus the curvature: the more curved the space is, the more massive the Gaussian fluctuations are.

Then we investigate the possible deviations away from this hyperbolic geometry: we prove that the Liouville field satisfies  a large deviation principle with an explicit good rate function, the Liouville action given by \eqref{actiondefolieetbut}. This rate function is non trivial, admits a unique 
minimum at  the solution to the Liouville equation with curvature $8\pi^2\Lambda$. At first sight and in view of the literature on the discrete GFF, one could naively infer that this result is not surprising.  However, the renormalization procedure in the definition of the interaction term in \eqref{laplace}  involves  hidden divergences, which make our statement a priori not necessarily natural. The proof is based on computing the exact asymptotic expansion of   \eqref{laplace} when $G$ is the exponential of a linear function of $\varphi$: in fact, this exact expression (see \eqref{exactequivalent} below when $G=1$) is the main result of this paper as all the other results stem from this
 asymptotic equivalent.

We must mention here that there is a sizeable literature on mathematical studies of discrete random surfaces in all dimensions: see for instance the recent review of Funaki 
\cite{Funaki}. In particular, within the (discrete) framework of gradient perturbations of the GFF, there has been an impressive series of results: large deviation principles (see \cite{DGI}) or central limit theorems (see \cite{NaSpenc} for a convergence to the massless GFF in the whole space) for instance. Nonetheless, the results of this paper bear major differences with the discrete case and, in particular, cannot be derived from the discrete frameworks previously considered. Indeed, in the discrete case, one can work in nice functional spaces whereas, in the continuum setting of this paper, the GFF lives in the space of distributions and not in a functional space. Also, since the Liouville measure does not live in a fixed Wiener chaos, one can not rely on classical estimates in the constructive field theory literature in order to get exponential approximations (see \cite{dembo} for the definition) of the continuum setting by the discrete setting. Besides, an important aspect of our work is the derivation of sharp
asymptotics for the partition function (see \eqref{exactequivalent} below) and more generally the Laplace transform of the field $\varphi$: this is a specific feature of the continuum setting which is essential in the problem of establishing exact relations for the three point correlation function, 
the celebrated DOZZ formula (Dom, Otto, Zamolodchikov and Zamolodchikov) \cite{DO,ZZ} derived on the sphere (see \cite{witten} for a recent article on this problem).

\begin{figure}[h] 
\centering
\includegraphics[width=0.5\linewidth,height=0.25\textheight]{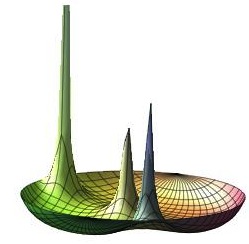}
\caption{Surface with negative curvature and conical singularities}
\label{pikedbob}
\end{figure}

Finally, we investigate the Liouville action with heavy matter insertions. This means that we plug $p$ exponential terms of the form $e^{\frac{\cX_i}{\gamma}X}$ (with $\cX_i\in [0,2[$) in the Liouville action in order to compute the $p$-point correlation functions of LFT. The terminology ``heavy matter'' is related to the fact that the exponential weights  $e^{\frac{\cX_i}{\gamma}X}$ are chosen in a regime so as to affect the saddle points of the action. We prove that the Liouville field $\gamma \varphi$ then concentrates on the solution of the Liouville equation with sources (where $\delta_{z_i}$ stands for the Dirac mass at $z_i$) 
 \begin{equation}\label{intro:source}
\triangle_{\hat{g}} U-R_{\hat{g}} =8\pi^2\Lambda e^U- 2\pi \sum_{i=1}^p \cX_i\delta_{z_i}\quad \quad U_{|\partial\U}=0.
\end{equation}
This equation appears when one looks for a metric with prescribed negative curvature $8\pi^2\Lambda$ and conical singularities at the points $z_1,\dots,z_p$ (see Figure \ref{pikedbob}). Each source $\cX_i\delta_{z_i}$ creates a singularity with shape $\sim \frac{1}{|x-z_i|^{\cX_i}}$ in the metric $e^{U(x)}\hat{g}$. Such singularities are called conical as they are locally isometric to a cone with 
``deficit angle'' $\pi \cX_i$ (see Figure \ref{fig:cone}).
 
\begin{figure}[h] 
\begin{center}
\begin{tikzpicture}[scale=1.5] 
\draw (-20:1) arc (-20:240:1) ;
\draw  (-20:1) -- (0:0) -- (240:1);
\draw[dashed,<->,>=latex] (240:0.4) arc (240:340:0.4);
\draw (290:0.4) node[below]{$\pi\cX$};
\draw[->,>=latex] (2,0) to[bend left] (3,0);
\draw (4,-0.5) .. controls (4,-1) and (6,-1) .. (6,-0.5) -- (5,1.5) --(4,-0.5);
\draw[dashed] (4,-0.5) .. controls (4,0) and (6,0) .. (6,-0.5);
\end{tikzpicture}
\end{center}
\caption{Cone with deficit angle $\pi\cX$. Glue isometrically the two boundary segments of the left-hand side figure to get the cone of the right-hand side figure. Such a cone is isometric to the complex plane equipped with the metric $ds^2=|z|^{-\cX}dzd\bar{z}$.}
\label{fig:cone}
\end{figure}
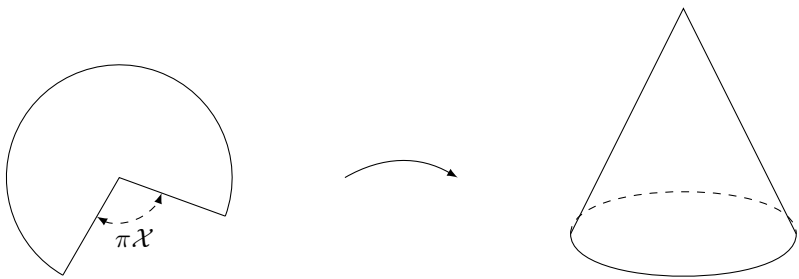

Here again,  the leading order fluctuations around this hyperbolic geometry with conical singularities  are Gaussian and massive in the sense that the rescaled field $ \varphi-\gamma^{-1}U$ converges towards a massive free field in the metric $e^{U(x)}\hat{g}$, where $U$ is now the solution of \eqref{intro:source}. The  mass of this free field is once again minus the curvature, namely $8\pi^2\Lambda$. We also establish a large deviation principle with an explicit good rate function, which is non trivial and admits a unique 
minimum at the solution of the Liouville equation with sources.

\vspace{2mm}

\noindent {\bf Conformal gravity in $4d$.} Let us stress that an analog $4d$-conformal field theory has been studied in the physics literature (see \cite{hamada}) from quantized gravity. The dynamics are governed by the Wess-Zumino action and the Weyl action. Basically, the underlying idea is that the $4d$ Paneitz operator is conformally covariant and yields a notion of $Q$-curvature. To put it simply, we can consider the Euclidean background metric so that the Paneitz operator simply becomes the bilaplacian. The action then becomes 
\begin{equation}\label{actionWZW}
S_{WZW}(\varphi)=\frac{1}{16\pi^2}\int_D \big[\langle  \triangle\varphi,  \triangle\varphi\rangle +16\pi^2\mu e^{\gamma\varphi}\big]\,\lambda(\dd x),
\end{equation}
which is the $4d$ analog of \eqref{actiondefolieetbut} in $4d$ in flat background metric.
The important point for our purposes is that the corresponding free field action ($\mu=0$) generates a log-correlated Gaussian field so that our approach applies word for word. 
We mention that such a theory shares fractal properties similar to $2d$ Liouville field theory, like the geometrical KPZ formula, as proved in \cite{BJRV,Rnew10}. The semiclassical limit is described in terms of the equation 
\begin{equation}\label{eqWZW}
\triangle^2=\Lambda e^U,
\end{equation} which is a prescription of constant (negative) $Q$-curvature. The reader may consult \cite{djadli} for more on this topic of $Q$-uniformization of $4d$-surfaces. Heavy matter operators may be added as well, leading to a perturbed equation \eqref{eqWZW} with additional sources (i.e. Dirac masses). This approach can also be generalized to even larger dimensions by considering the
conformally covariant GJMS operators (from  Graham, Jenne, Mason and Sparling \cite{cf:GJMS}), which take the simple form $\triangle^{d/2}$ in even $d$-dimensional flat space.
\vspace{2mm}

\noindent {\bf Discussion  on possible extensions or other geometries.}
 Extra boundary terms $$\frac{1}{2\pi}\oint_{\partial D} \big[Q r_{\hat{g}}  \varphi +\kappa e^{\gamma \varphi/2}\big]\,d\ell,$$
where $d\ell$ stands for the length element on $\partial D$ in the metric $\hat{g}$ and $r_{\hat{g}}$ for the geodesic curvature on $\partial D$, may be considered as well in the Liouville action \eqref{actiondefolieetbut}. These boundary terms rule the behaviour of the Liouville field on the boundary and gives rise to the boundary Liouville Field theory. We will refrain from considering these extra terms here: our purpose 
is to expose some aspects of  LFT mainly to mathematicians and we wish to avoid these further complications. Yet, this would be a first natural (and non trivial) extension of our work.

In this paper,  we focus on the unit disk with flat geometry but we stress that hyperbolic geometry can be handled the same way. Of particular interest is the construction of LFT on the sphere \cite{DKRV}. In that case, the semi-classical limit exhibits some further interesting features. The point is that the limiting equations requires to construct a hyperbolic structure on the sphere, which is rather not inclined to support such a structure. This can be addressed by taking care of the nature of the insertions in the surface: there are some additional constraints on the insertions $(z_i,\cX_i)_i$, which are called  Seiberg bound in the physics literature. Yet this framework will be addressed elsewhere. This problem  also receives a new growing interest in the community of differential geometry: the reader may consult for instance \cite{bartolucci,battaglia,troyanov} and references therein for more on this topic and other closely related topics, like the Toda system. Indeed, another natural extension of our work could be to consider the large deviations of Toda field theories. In fact, tilting the free field measure with any nonlinear functional of the free field that yields interesting critical points for the Laplace method deserves to be investigated.

\subsection*{Acknowledgements}
The authors wish to address special thanks to Fran\c{c}ois David. This work originates from one of the numerous discussions we have had with him. The authors are also very indebted to Andrea Malchiodi and Yannick Sire who  patiently explained to them how to deal with equation \eqref{intro:source}.

\section{Background and notations}\label{sec:setup}
\subsection{Notations}

{\bf Differential geometry:} The standard gradient,  Laplacian and Lebesgue measure on (a subdomain of) $\R^2$   are denoted by $\partial$, $\triangle$ and $\lambda(\dd x)$ (and sometimes even $\dd x$). We will adopt the following notations related to Riemannian geometry throughout the paper.  On a bounded domain $D$ of $\R^2$,  a smooth function $\hat{g}:D\to ]0,\infty[$ defines  a scalar metric tensor by 
$$(x,u,v)\in D\times \R^2\times\R^2\mapsto \hat{g}(x)\langle u,v\rangle,$$
where $\langle u,v\rangle$ stands for the canonical inner product on $\R^2$. In what follows, we will denote by $\hat{g}(x) \dd x^2$ this metric tensor and sometimes, with a slight abuse of notation, identify $\hat{g}(x) \dd x^2$ with the function $\hat{g}$.  

We can associate to this metric tensor a gradient $\partial^{\hat{g}}$, a Laplace-Beltrami operator $\triangle_{\hat{g}}$, a Ricci scalar curvature $R_{\hat{g}}$,  and a volume form $\lambda_{\hat{g}}$, which are defined by:
\begin{align}
\partial^{\hat{g}} \varphi(x)=&\hat{g}(x)^{-1}\partial \varphi (x)& & R_{\hat{g}}(x)=-\triangle_{\hat{g}}\ln \hat{g}(x)\\
\triangle_{\hat{g}} \varphi(x)=& \hat{g}(x)^{-1}\triangle \varphi (x) & & \int_D\varphi(x)\,\lambda_{\hat{g}}(\dd x)= \int_D\varphi(x) \hat{g}(x)\,\lambda(\dd x).
\end{align}

We denote by $\langle  \partial^{\hat{g}}\varphi, \partial^{\hat{g}}\psi\rangle_{\hat{g}}$ the pairing of two gradients $ \partial^{\hat{g}}\varphi$, $\partial^{\hat{g}}\psi$  in the metric $\hat{g}$, that is
\begin{equation}
\langle  \partial^{\hat{g}}\varphi, \partial^{\hat{g}}\psi\rangle_{\hat{g}}(x)= \hat{g}(x)^{-1}\langle \partial \varphi(x),\partial \psi(x)\rangle.
\end{equation}

\vspace{1mm}

{\bf Green function and conformal maps:} The Green function on a domain $D$ will be denoted by $G_D(x,y)$. By definition, the Green function is the unique function which solves the following equation for all $x \in D$:
\begin{equation*}
\triangle_y G_D(x,y)= - 2 \pi \delta_{x}, \; G_D(x,.)=0 \text { on }\partial D .
\end{equation*} 
Note that with this convention $G_D(x,y)= \ln \frac{1}{|y-x|}+\varphi(x,y)$ where $\varphi$ is a smooth function on $D$ (not smooth on the whole boundary $\partial D$). By conformal map $\psi:\tilde{D}\to D$, we will always mean a bijective bi-holomorphic map from $\tilde{D}$ onto $D$. Recall that the Green function is conformally invariant in the sense that $G_{D} \circ \psi= G_{\tilde{D}}$. 

\vspace{1mm}

{\bf Functional spaces:} $C^\infty_c(D)$ stands for the space of smooth compactly supported functions on $D$. We denote $\mathbb L^p(D)$ the standard space of functions $u$ such that $|u|^p$ is integrable. Classically, if $D$ is  a (say) smooth bounded domain, we define the space $H^1_0(D)$ as the completion of $C^\infty_c(D)$ with respect to the (squared) norm $ |\varphi|_{H^1}^2=\int_D|\partial \varphi|^2\,dx$. Let us recall a few facts on $H^1_0(D)$ and its dual $H^{-1}(D)$ which we need in the paper: see \cite[section 4.2]{dubedat} for instance. The space $H^{-1}(D)$ is defined as the Banach space of  continuous linear functionals $f$ on $H^1_0(D)$ equipped with the norm
\begin{equation*}
|f|_{H^{-1}}= \underset{\varphi \in H^1_0(D), \: |\varphi|_{H^1} \leq 1 }{\sup} f(\varphi)  
\end{equation*}      
where we denote $f(\varphi)  $ the distribution $f$ applied at $\varphi$. The dual space of $(H^{-1}(D), |.|_{H^{-1}})$ is once again a Banach space. The space $H^1_0(D)$ can then be equipped with the weak$^\star$ topology, i.e. the topology induced by the linear functionals $\varphi\in H^1_0(D) \mapsto f(\varphi)$ for all $f$ in $H^{-1}(D)$: see \cite[Appendix B]{dembo}. This topology coincides with the standard weak topology on $H^1_0(D)$. We will use this remark when establishing the large deviation principle.

%

\subsection{Gaussian Free Fields}\label{sec:cut}
 Let $\P,\E$ denote the probability law and expectation of a standard probability space; the corresponding space of variables $Z$ such that $|Z|^p$ is integrable will be denoted by $\mathbb L_p$. On this space, the centered Gaussian Free field  $X$ with mass $m\ge 0$ on a planar domain $D\subset \R^2$  and Dirichlet boundary condition
is the Gaussian field whose covariance function is given by the Green function $G_D^m$ (recall that, for $m\equiv 0$, we denote $G_D=G_D^0$) of the problem
$$\triangle u-  mu=-2\pi f \text { on }D,\quad u_{|\partial D}=0,$$ where $m\geq 0$ is a function defined on $D$. When the mass satisfies $m\not\equiv 0$, one usually talks about Massive Free Fied (MFF for short) whereas one rather uses the terminology Gaussian Free Field (GFF for short) for the massless field with $m\equiv 0$. Therefore, for any smooth compactly supported   functions $f,h$ on $D$ 
 $$\E\Big[X (f)X (h)\Big]=\iint_{D\times D} f(x)G^m_D(x,y)h(y) \dd x\dd y.$$
Almost surely, the GFF lives on the space $ H^{-1}(D)$ (see \cite{dubedat}). 

\begin{rem}
In fact, $X$ belongs to the standard Sobolev space $H^{-s}(D)$ for all $s>0$ (see \cite{dubedat} for further details) but for simplicity, we refrain from considering this framework. Many theorems of this paper could in fact be strengthened to the topology of $H^{-s}(D)$; for instance, this is the case for the large deviation result, i.e. Theorem \ref{th:ldp}, by using \cite[Theorem 4.2.4]{dembo} which enables to strengthen topologies in large deviation principles.    

\end{rem}
  
\begin{rem}
We could treat other boundary conditions as well but for simplicity, we restrict to the case of Dirichlet boundary conditions. In the case when the action possesses boundary terms, it is more relevant to consider a GFF with Neumann boundary conditions in the following.
\end{rem}

In what follows, we need to consider   cutoff approximations of the Gaussian free field $X$ on $D$.  The cutoff may be any of the following: 
\begin{description}
\item[-White Noise (WN)] (see \cite{Rnew10,lacoin,review}): The Green function $G_D$ on $D$ can be written as 
$$G_D(x,y)=\pi\int_0^\infty   p(r,x,y)\,dr.$$
where $p(t,x,y)$ will denote the transition densities of the Brownian motion on $D$ killed upon touching $\partial D$.
A formal way to define   the  Gaussian field $X $   is to consider  a white noise  $W$ on $\R_+\times D$ and define
\begin{equation}\begin{split}\label{GFFXY}
X(x)&=\sqrt{\pi}\int_{0}^{\infty}\int_{D}  p(r/2,x,y)\,W^X(dr,dy).
\end{split} \end{equation}
We define the approximations $X_{\gep}$  by integrating over $(\gep^2,\infty)\times D$ in \eqref{GFFXY} instead of $(0,\infty)\times D$. The covariance function for these approximations is given by 
\begin{equation}\label{covar1}
\E[X_\gep(x)X_{\gep'}(y)]=\pi\int_{\gep^2\vee\gep'^2}^{\infty}  p(r,x,y)\,dr.
\end{equation}
\item[-Circle Average (CA)] (see \cite{cf:DuSh}): We introduce the circle averages $(X_\gep)_{\gep\in]0,1]}$   of radius $\gep$, i.e. $X_\gep(x)$ stands for the mean value of $X$ on the circle centered at $x$ with radius $\gep$. We could also consider more general mollifiers (see \cite{legrosrobert,RV}).
\item[-Orthonormal Basis Expansion (OBE)]  (see  \cite{dubedat,cf:DuSh,lacoin,review}): We consider an orthonormal basis $(f_k)_{k\geq 1}$ of $H^1_0(D)$ made up  of continuous functions and the projections of $X$ onto this orthonormal basis, namely we define the sequence of i.i.d. Gaussian random variables:
\begin{equation*}
\gep_k= \frac{1}{2\pi}\int_D \langle\partial X(x), \partial f_k(x)\rangle dx.
\end{equation*}
The projections  of $X$   onto the span of $\{f_1,\dots,f_n\}$ are given by $X_n(x)= \sum_{k=1}^n  \gep_k f_k(x)$.
\end{description}

In any of the above three cases, the family of cutoff approximations  will be denoted by $(X_\epsilon)_\epsilon$ (with $\epsilon=e^{-n}$ in the case of  (OBE)).

  \begin{rem}
      Another way to obtain an approximation of the Gaussian  Free Field, is to consider first a lattice version.
      We do not mention it here as it is slightly less convenient than the other options in our setup: 
      the lattice field and the continuous field are not defined on the same space
      and it is quite technical to construct a relevant coupling between the two of them.
      \end{rem} 

\subsection{Gaussian multiplicative chaos}\label{sub:chaos}
For the three possible cutoff approximations $(X_\epsilon)_\epsilon$ of the GFF and for $\gamma\in [0,2[$, we consider the random measure on $D$ defined by 
\begin{equation}\label{mido}
 e^{\gamma X(x)}\,dx=\lim_{\epsilon\to 0}\epsilon^{\frac{\gamma^2}{2}} e^{\gamma X_\epsilon (x)}\,dx.
\end{equation}
The limit holds almost surely and is understood in the sense of weak convergence of measures. This has been proved in \cite{cf:Kah} for the cutoff family (WN) and (OBE) and in \cite{cf:DuSh}  for (CA). The limit is non trivial if and only if $\gamma<2$ (see \cite{cf:Kah}). For these three possible cutoffs, the limiting objects $(X,e^{\gamma X(x)}\,dx)$ that
we get by taking the limit as $\epsilon\to 0$ have the same law \cite{review}.

\subsubsection{The Wick Notation}\label{wick}
In the paper we make extensive use of Wick notation for the exponential.
If $Z$ is a Gaussian variable with mean zero and variance $\sigma^2$, its Wick $n$-th power ($n\in \N$) is defined by 
\begin{equation}
:Z^n: \,\,= \sum_{m=0}^{\lceil n/2 \rceil}\frac{(-1)^m n!}{m!(n-2m!)2^m} \sigma^{2m} Z^{n-2m}= \sigma^n H_n(\sigma^{-1} Z)
\end{equation}
where $H_n$ is the $n$-th Hermite Polynomial. 
If $Z$ is not centered then $:Z^n:$ is understood as $:Z^n:\,\,=\,\, :\tilde Z^ n:$ where $\tilde Z:= Z-\E[Z]$.


\medskip

This definition is designed to make  the Wick monomials orthogonal to each other. More precisely if $(Z,Y)$ is a Gaussian vector we have
\begin{equation}\label{scalar}
\E\left[:Z^n: :Y^m: \right] = n! \ind_{n=m}\E\left[ZY\right]^n.
\end{equation}


The Wick exponential is defined formally as the result of the following expansion in Wick powers
\begin{equation}\label{supertaylor}
:e^{\gamma Z}:\,\, = \sum_{n=0}^{\infty} \frac{\gamma^n :Z^n:}{n!}.
\end{equation}
A bit of combinatorics with Wick monomials leads to the following identity
\begin{equation}\label{wickexpo}
:e^{\gamma Z}:= \exp\left( \gamma \tilde Z-\frac{ \sigma^2 \gamma^2}{2}\right).
\end{equation}
Most of the time we will use the Wick notation for Gaussian fields that are distributions rather than Gaussian variables, but we specify the meaning of this notation below.

\subsubsection{Wick Notation for Gaussian Fields}
We consider    a Free Field $X$ defined on a planar domain $D$ (we stress that the basics below extend without changes to any other log-correlated Gaussian field). 
We define the Wick powers and the Wick exponential as a distribution on $D$, by taking the limit of
cutoff approximations of $X$ constructed in Section \ref{sec:cut}. Indeed from the formula \eqref{scalar}, 
the reader can check that for any smooth function $u$, and any $n$
the sequence indexed by $\gep>0$

\begin{equation}
\int_D :X^n_\gep(x): u(x) \dd x
\end{equation}
is Cauchy in $\mathbb L_2$, and thus admits a   limit when $\gep \to 0$.
The limit is denoted by $\langle :X^n:,u \rangle$. With some additional work
one can check that $:X^n:$ defines a random distribution acting on $u$.
%
%
For $\gamma<2$, one can also consider the limit $:e^{\gamma X(x)}:\,dx$ in the sense of weak convergence of measures of the family $(:e^{\gamma X_\epsilon(x)}:\,dx)_\epsilon$ and one can check that
\begin{equation}
e^{\gamma X(x)}\,dx= \,\,:e^{\gamma X}:\, C(x,\U)^{\gamma^2/2}dx,
\end{equation}
where $C(x,\U)$ denotes the conformal radius and the measure is defined in subsection \ref{sub:chaos}.
Notice that for $ \gamma<\sqrt{2}$, the limit can also be obtained from the series expansion \eqref{supertaylor}: for all $u\in \mathbb L^p(D)$ for some $p>1$
\begin{equation}\label{eqdefconvergence}
\int_D :e^{\gamma X(x)}: u(x)\,\dd x=\sum_{n\geq 0} \frac{\gamma^n}{n!} \int_D  :\! X^n(x)\! : u(x)\,\dd x,
\end{equation}
where the above sum converges in $\mathbb{L}_2$.

%

\section{Semiclassical limit}\label{sec:semi}
\subsection{The semiclassical limit}

In this section and for pedagogical purpose, we make one simplification by not taking into account a possibly curved space. As mentioned in introduction, the case of the hyperbolic geometry can be handled the same way whereas additional
difficulties appear on the sphere, which the methods of this paper cannot handle.   

We equip the unit disk $\U$ with the flat metric, i.e. the metric associated to the metric tensor $g=1$ on $\U$. We consider a GFF $X$ on $\U$ with Dirichlet boundary condition. We consider a cosmological constant $\mu\geq 0$ and a Liouville conformal factor $\gamma\in]0,2]$. We set 
$$Q=\frac{2}{\gamma}+\frac{\gamma}{2}.$$

We define the law $\P_{\mu,\gamma}$ of the Liouville field  $X$ on $\U$ associated to $(\mu,\gamma)$ as the tilted version of $\P$
as follows:
\begin{equation}\label{defliouville}
\E_{\mu,\gamma}[F(X)]=Z^{-1}_{\mu,\gamma}\E\Big[F(X)\exp\Big(-4\pi \mu\int_{{\U}}e^{\gamma X(x)} \dd x\Big)\Big]
\end{equation}
where  
$$  Z_{\mu,\gamma}=\E\Big[\exp\Big(-4\pi \mu\int_{\U}e^{\gamma X(x)}\dd x \Big)\Big]$$ and $F$ is any bounded continuous functional on $H^{-1}(\U)$.

\medskip

Our aim is to determine the asymptotic behavior of the fields $\gamma X$ when $\gamma$ tends to zero and $\mu$ tends to infinity simultaneously
while satisfying the relation
\begin{equation}\label{asymp}
\mu \gamma^2=\Lambda,
\end{equation}
for a fixed positive $\Lambda$.

\medskip

We claim 
\begin{theorem}\label{th:semi} 
Assume that $\gamma\to 0,\mu\to\infty$ under the constraint \eqref{asymp}. The field  $\gamma X$  concentrates on the solution of the classical Liouville equation
\begin{equation}\label{eq:class}
\triangle U =8\pi^2\Lambda  e^U
\end{equation}
 with zero boundary condition on $\U$.   More precisely
\begin{enumerate}
\item The partition function has the following asymptotic behavior at the exponential scale  
\begin{equation}
\lim_{\gamma\to 0} \gamma^2 \log Z_{\mu,\gamma}= -  \frac{ 1}{4\pi  } \int_\U ( |\partial U(x)|^2+16\pi^2\Lambda e^{ U(x)}) \dd x =: \tf(\gL).
\end{equation}
\item More precisely we have the following equivalent as $\gamma\to 0$
\begin{equation}\label{exactequivalent}
Z_{\mu,\gamma}\sim e^{\gamma^{-2} \tf(\gL)}
\exp\left(-  2\pi\Lambda  \int_\U e^{ U(x)}\ln C(x,\U) \dd x \right)  \E\left[ \exp\left(-2\pi\Lambda\int_\U  e^{U(x)}  : X(x)^2 : \dd x   \right)  \right ],
\end{equation}

 where $:X^2:$ is the standard Wick-ordered square field, i.e. $:X(x)^2:= \underset{\epsilon \to 0}{\lim} \; X_\epsilon(x)^2- \E[X_\epsilon(x)^2 ]$ where $X_\epsilon$ is the cutoff field (see subsection \ref{sec:cut}).
\item The field $\gamma X$ converges in probability in  $ H^{-1}(\U)$   as $\gamma\to 0$  towards $U$.
\item Both  random measures $:e^{\gamma X}: \dd x$ and $e^{\gamma X} \dd x$ converge in law in the sense of weak convergence of measures towards $e^{U(x)}\,±\dd x$ as $\gamma\to 0$.
\item  the field $X-\gamma^{-1}U$ converges in law in  $H^{-1}(\U)$  as $\gamma\to 0$ towards a Massive Free Field in the metric $\hat{g}=e^{U(x)}\dd x^2$ with Dirichlet boundary condition and mass $8\pi^2\Lambda$ , that is a Gaussian field with covariance kernel given by the Green function of the operator $2\pi (8\pi^2\Lambda-\triangle_{\hat{g}})^{-1}$ with Dirichlet boundary conditions.
\end{enumerate}
 
\end{theorem} 

\begin{rem}
The above theorem shows in a way that the metric $e^{\gamma X(x)}\dd x^2$ converges  as $\gamma\to 0$ towards the metric on the disk with negative curvature $-8\pi^2\Lambda$. Actually, we only treat here the case of the curvature or volume form of the metric. But the same argument can be adapted for instance to prove the convergence of the associated Brownian motion defined in \cite{GRV,GRV-FD}.
\end{rem}

\subsection{The large deviation principle}

Now we focus on a Large Deviation Principle. Recall that $U$ is the solution to the classical Liouville equation \eqref{eq:class}. For $f\in H^1_0(\U)$, we consider the  weak solution $V$ of the perturbed Liouville equation (see Theorem \ref{th:liouville1})
\begin{equation}\label{boucherie}
\triangle V=8\pi^2\Lambda e^{V(x)}-2\pi   f(x), \quad \text{with }V_{|\partial \U}=0,
\end{equation}
and we set  
$$\tf  (\Lambda,f)=  -  \frac{ 1}{4\pi  } \int_\U ( |\partial V(x)|^2+16\pi^2\Lambda e^{ V(x)}) + \int_\U f(x)( V(x)-U(x))\dd x.$$ 
In the course of the proof of our large deviation result: Theorem \ref{th:ldp}, we will check that the mapping $f \in H^1_0(\U)\mapsto \tf(\Lambda,f)-\tf(\gL)$ is convex, G\^ateaux-differentiable and weakly lower semi-continuous (for the weak$^{\star}$ topology). 

\medskip

We define its Fenchel-Legendre transform as follows  by 
\begin{equation}\label{gignac}
\forall h\in H^{-1}(\U), \quad I^*(h)=\sup_{f\in H^{1}_0(\U)}\{h(f)-\tf  (\Lambda,f)+\tf  (\Lambda)\}.
\end{equation}
\begin{proposition}\label{ratefunction}
The function $I^*$ is a good rate function with explicit expression 
\begin{equation*}
I^*(h)=\begin{cases} 
E(U+h)-E(U), &\text{if } h \in H^1_0(\U)  ,\\
+\infty,& \text{otherwise},
\end{cases}
\end{equation*} where
$$\forall u\in H^1_0(\U),\quad E(u)=  \frac{ 1}{4\pi  } \int_\U ( |\partial u(x)|^2+16\pi^2\Lambda e^{ u(x)})\dd x.$$
In particular, we have $I^*(h)>0$ except for $h=0$.
\end{proposition}
%

The fact that $I^*$ vanishes only for $h=0$  is important because this entails that the forthcoming LDP provides non trivial bounds as soon as the set $A$ has non empty interior and $0\not\in \overline{A}$.

\begin{theorem}\label{th:ldp} 
Assume that $\gamma\to 0,\mu\to\infty$ under the constraint \eqref{asymp}. 
Set $Y_\gamma=\gamma X-U$. The following LDP holds with good rate function $I^*$ on the space $H^{-1}(\U)$ equipped with the norm $|.|_{H^{-1}}$
$$-\inf_{h\in \mathring{A}}I^*(h)\leq  \liminf_{\gamma\to 0}\gamma^2\P_{\mu,\gamma}(Y_\gamma\in A)\leq   \limsup_{\gamma\to 0}\gamma^2\P_{\mu,\gamma}(Y_\gamma\in A)\leq -\inf_{h\in \overline{A}}I^*(h)$$ for each Borel subset $A$ of $H^{-1}(\U)$.
\end{theorem}

\subsection{ Proof of Theorem \ref{th:semi}}
 We first compute the limit of the partition function $Z_{\mu,\gamma}$. 
\begin{align*}
  \E\big[\exp\left(- \frac{4\pi\Lambda}{\gamma^2} \int_\U e^{\gamma X(x)} \, \dd x \right) \big]   &=\E\big[\exp\left( - \frac{4\pi\Lambda}{\gamma^2} \int_\U :e^{\gamma X(x)}:C(x,\U)^{\frac{\gamma^2}{2}}\, \dd x \right)\big] .
\end{align*}
We define $Y=Y_\gamma$ as follows
\begin{equation}\label{defy}
Y(x)=X- \frac{U}{\gamma}
\end{equation}
Where $U$ is the solution of \eqref{eq:class}. Note that this implies in particular that
\begin{equation}\label{alternative}
 U(x)= -4\pi \Lambda \int_\U e^{U(y)} G_\U(x,y)  \dd y.
 \end{equation}
We have 
\begin{equation}\label{vampetta}\begin{split}
 & \E\left[\exp\left(- \frac{4\pi\Lambda}{\gamma^2} \int_\U e^{\gamma X(x)} \, \dd x \right)\right]  \\
  &=\E\left[\exp\left(-\frac{4\pi\Lambda}{\gamma^2} \int_{\U} e^{U(x)}(1+\gamma Y(x)) \, \dd x \right) 
  \exp\left( - \frac{4\pi\Lambda}{\gamma^2} \int_\U e^{\gamma X(x)}-e^{U(x)}(1+\gamma Y(x)) \, \dd x \right)\right]  \\
 &= \exp\left(-  \frac{4\pi\Lambda}{\gamma^2} \int_\U e^{ U(x)}(1-U(x)) 
 \dd x+ \frac{8\pi^2\Lambda^2}{ \gamma^2}\int_{\U^2}  e^{ U(x)+U(y)}  G_\U(x,y)\dd x \dd y \right)\\
 & \quad \quad \quad \quad \quad \times
 \E\bigg[ \exp\left(- \frac{4\pi\Lambda}{\gamma } \int_\U  e^{ U(x)} X(x) \dd x- \frac{8\pi^2\Lambda^2}{ \gamma^2}   \int_{\U^2}  e^{ U(x)+U(y)}
 G_\U(x,y)\dd x \dd y\right)\\
 & 	\quad \quad \quad \quad \quad \quad \quad \quad  \quad \quad \quad \quad  \quad \quad \quad \quad \quad \quad 
 \exp\left(- \frac{4\pi\Lambda}{\gamma^2} \int_\U e^{\gamma X(x)}-e^{U(x)}(1+\gamma Y(x)) \, \dd x \right)\bigg].
\end{split}
\end{equation}
The first exponential term in the expectation 
$$e^{- \frac{4\pi\Lambda}{\gamma } \int_\U  e^{ U(x)} X(x) \dd x- \frac{8\pi^2\Lambda^2}{ \gamma^2}   \int_{\U^2}  e^{ U(x)+U(y)}  G_\U(x,y)\dd x \dd y} 
 $$
is a Cameron-Martin transform term. It has the effect of shifting the field $X$ by a function which is equal to 
\begin{equation}
-\frac{4\pi\Lambda}{\gamma }\int_\U  e^{ U(x)} G_\U(x,y)\dd x=\frac{U(x)}{\lambda}.
\end{equation}
Hence after this shift, $Y$ becomes a centered field, and the expectation in the last line of \eqref{vampetta} is equal to 
\begin{align*}
\E\Big[   \exp\left(- \frac{4\pi\Lambda}{\gamma^2} \int_\U e^{ U(x)} (e^{\gamma X(x)}-1- \gamma X(x) )\dd x \right)\Big] .
\end{align*} 
For the term in front of the expectation, from \eqref{alternative} we have the following simplification
\begin{align*}
8\pi^2 \Lambda ^2\int_{\U^2}  e^{ U(x)+U(y)}  G_\U(x,y)\dd x \dd y=&-2\pi\Lambda \int_{\U}  e^{U(x)} U(x)\dd x=-\frac{1}{4\pi}\int_{\U} U(x) \triangle U(x)  \dd x\\=&\frac{1}{4\pi}\int_{\U} |\partial U(x)|^2   \dd x,
\end{align*}
and thus it is equal to  
\begin{equation}
\exp\left(-  \frac{ 1}{4\pi \gamma^2} \int_\U ( |\partial U(x)|^2+16\pi^2\Lambda e^{ U(x)}) \dd x\right).
\end{equation}

The computation of the partition function (item 1. and item 2.) is  completed with the following lemma.
\begin{lemma}\label{lem:neymar}
The quantity $Z_\alpha=\E[e^{-\alpha\int_\U  e^{U(x)}  : X(x)^2 : \dd x   }  ]$ is finite for any $\alpha>0$ and under the probability measure
$$\tP=Z_\alpha^{-1}e^{-\alpha\int_\U  e^{U(x)}  : X(x)^2 : \dd x   }  $$ the field $X$ has the law of a Massive Free Field in the metric $g=e^{U(x)}\dd x^2$ with Dirichlet boundary condition and mass $4\pi\alpha$ , that is a Gaussian field with covariance kernel given by the Green function of the operator $2\pi (4\pi \alpha-\triangle_g)^{-1}$.
\end{lemma}

\noindent {\it Proof of Lemma \ref{lem:neymar}.} Let $\hat{g}$ be the metric tensor $e^{U(x)}\,\dd x^2$.  Let $(\lambda_j)_j$ be the non-decreasing sequence of eigenvalues of $-(2\pi)^{-1} \Delta_{\hat{g}}$ with Dirichlet boundary conditions (with repetition if necessary to take into account multiple eigenvalue). Let $(e_j)_j$ be an orthogonal sequence of eigenvectors
associated to $\lambda_j$ normalized to $1$ in the $\mathbb{L}^2(\U,\lambda_{\hat{g}})$ sense, i.e. $\int_\U e_j(x)^2 \lambda_{\hat{g}}(\dd x)=1$. 
Note that the sequence $(e_j)_j$ is orthogonal in $\mathbb{L}^2(\U,\lambda_{\hat{g}})$  and in the Sobolev space $H^1_0(\U)$ (see \cite[chapter 7]{buser}).
Recall that $\lambda_j \sim   C j$ as $j$ goes to infinity according to 
 Weyl's asymptotic formula \cite{Weyl1,Weyl2}.  Then we have
\begin{equation*}
X(x)= \sum_{j=1}^\infty \frac{e_j(x)}{\sqrt{\lambda_j}} \epsilon_j
\end{equation*}
where $(\epsilon_j)_j$ is an i.i.d. sequence of standard Gaussian variables given by $\epsilon_j=\langle X,e_j\rangle_{H^1}$. In this case, we have
\begin{align*}
\int_\U :X(x)^2:e^{U(x)}\dd x & = \int_\U \Big(  (\sum_{j=1}^\infty \frac{e_j(x)}{\sqrt{\lambda_j}} \epsilon_j)^2 -\sum_{j=1}^\infty \frac{e_j(x)^2}{\lambda_j}\Big)  e^{U(x)}\dd x   
   = \sum_{j=1}^\infty  \frac{\epsilon_j^2-1}{\lambda_j}  .
\end{align*}  
Therefore we have
\begin{align*}
Z_\alpha & = \E[   e^{-\alpha   \int_\U :X(x)^2:e^{U(x)}\dd x }  ]   =  \prod_j  \E [   e^{-\frac{\alpha}{\lambda_j}\epsilon_j^2 } ]  e^{ \frac{\alpha}{\lambda_j}}   =   \prod_j  \sqrt{\frac{\lambda_j}{\lambda_j+2\alpha}} e^{ \frac{\alpha}{\lambda_j}},
\end{align*}
which converges if $\alpha>- \frac{\lambda_1}{2}$. By computing the Laplace transform of the field $X$ under $\tP$, it is plain to see that the field is Gaussian. It remains to identify  the   covariance structure
\begin{align*}
\tE[X(x)X(y)]& =\E[   X(x) X(y) e^{-\alpha   \int_D :X(x)^2: e^{U(x)}\dd x }  ] /  Z_\alpha\\
& = \sum_{j=1}^\infty  \frac{e_j(x) e_j(y)}{\lambda_j}  \frac{\E[  \epsilon_j^2    e^{-\alpha  \sum_{k=1}^\infty  \frac{\epsilon_k^2-1}{\lambda_k} }   ]  }{Z_\alpha}  \\
& = \sum_{j=1}^\infty  \frac{e_j(x) e_j(y)}{\lambda_j} \sqrt{\frac{\lambda_j+2\alpha}{\lambda_j}} e^{-\frac{\alpha}{\lambda_j}} \E[  \epsilon_j^2    e^{-\alpha   \frac{\epsilon_j^2-1}{\lambda_j} }   ]    \\
& = \sum_{j=1}^\infty  \frac{e_j(x) e_j(y)}{\lambda_j+2 \alpha}.
\end{align*}
Hence the law of $X$ is that of the Massive Free Field in the metric $e^{U(x)}\dd x^2$ conditioned to be $0$ on the boundary of $\U$ with mass $2\alpha\times 2\pi$. \qed

\begin{lemma}\label{waddle}
For any bounded positive function $g$ on $\U$ one has
\begin{multline}\label{weah}
\lim_{\gamma\to 0} \E\Big[   \exp\left(-\frac{1}{\gamma^2} \int_\U g(x) (e^{\gamma X(x)} -1- \gamma X(x)) \dd x \right)\Big] 
\\ = e^{-\frac{1}{2}\int_{\U} \log C(x,\U)\dd x}\E[  \exp\left(-\int_\U  g(x)  : X(x)^2 : \dd x   \right)  ],
\end{multline}  where $:X^2:$ is the standard Wick-ordered square field defined in Section \ref{wick}.
\end{lemma}

\begin{rem}
Before to the proof of this lemma, let us make a few comments. First, observe  that the expression in the exponential is not positive as the elementary inequality $e^u-1-u\geq 0$ might suggest.  Indeed one should not forget that here $e^{\gamma X}\dd x$ is defined via a renormalization procedure (recall \eqref{mido}). 
This being clear, let us shortly explain why \eqref{weah} holds. We have
\begin{equation}\begin{split}
& \exp\left(-\frac{1}{\gamma^2} \int_\U g(x) (e^{\gamma X(x)} -1- \gamma X(x)) \dd x \right)\\
 &\quad=  \exp\left( \frac {1}{\gamma^2} \int_\U g(x) :e^{\gamma X(x)}: (1- \left(C(x,\U)\right)^{\gamma^2/2})\dd x\right) \\
&\quad\quad\quad\quad \times \exp\left(- \frac{1}{\gamma^2} \int_\U g(x) (:e^{\gamma X(x)}: -1- \gamma X(x)) \dd x \right).
\end{split}\end{equation}
When $\gamma$ tends to zero  $(1- \left(C(x,\U)\right)^{\gamma^2/2})$ is equivalent to $-\gamma^2/2 \log C(x,\U)$.
According to the expansion \eqref{eqdefconvergence}, it also makes sense to say that $:e^{\gamma X(x)}:\dd x \sim \dd x  $ in some sense as $\gamma$ goes to $0$ so that
the first  term should converge to $e^{-\frac{1}{2}\int_{\U} g(x) \log C(x,\U)\dd x}$.

\medskip

As for the second term \eqref{eqdefconvergence} tells us that 
\begin{equation}
\gamma^{-2} g(x) (:e^{\gamma X(x)}: -1- \gamma X(x)) \dd x \sim \frac{g(x)}{2}:X^2(x): \dd x
\end{equation}
which indicates convergence.

\medskip

The difficult part is then to show that the formal equivalent above are rigorous in a sense, and also that lower order terms do not change the behavior of the Laplace transform on the left-hand side of \eqref{weah}. 
\end{rem}

\noindent {\it Proof of Lemma \ref{waddle}.}  The first step is to prove that the random variable
$$H_\gamma=\frac{1}{\gamma^2} \int_\U (e^{\gamma X(x)} -1- \gamma X(x)-\frac{\gamma^2}{2}\ln C(x,\U)-\frac{\gamma^2}{2}  : X(x)^2 : )  g(x) \dd x$$
converges in $\mathbb L_2$ towards $0 $ as $\gamma\to 0$. We have

\begin{align}\label{eq:mbappe}
H_\gamma & =\frac{1}{\gamma^2} \int_\U :e^{\gamma X(x)}:\left( C(x,\U)^{\gamma^2/2}-1\right)- \frac{\gamma^2}{2}\ln C(x,\U)\dd x\\ 
&\quad \quad {\hub +} \frac{1}{\gamma^2} \int_\U (:e^{\gamma X(x)}: -1- \gamma X(x)-\frac{\gamma^2}{2}  : X(x)^2 : )  g(x) \dd x
\end{align} 
It is rather straightforward to check the convergence to zero in $\mathbb L_2$ of the first term.
As for the term of the second line, by the expansion \eqref{eqdefconvergence} and the orthogonality of Wick polynomials \eqref{scalar}, its variance is equal to

\begin{equation}
\sum_{n \ge 3}  \frac{1}{n!}\gamma^{2n-4} \int_{\U^2}   (G^n_\U(x,y))^n g(x)g(y)  \dd x \dd y.
\end{equation}
 The variance being of order $\gamma^{2}$ the second line of \eqref{eq:mbappe} is of order $\gamma$. 

 The second step is to prove a form of tightness.  Indeed if we can prove that for any positive bounded function $g$ one has  
\begin{equation}\label{zidane}
\sup_{\gamma>0} \E\Big[   e^{-\gamma^{-2}\int_\U 2g(x) (e^{\gamma X(x)} -1- \gamma X(x)-\frac{\gamma^2}{2}\ln C(x,\U)) \dd x }\Big] <+\infty. 
\end{equation}
 then the first step implies the following convergence in probability
$$e^{-\gamma^{-2}\int_\U g(x) \left((e^{\gamma X(x)} -1- \gamma X(x)-\frac{1}{2}\ln C(x,\U) \right)  \dd x}\to   e^{-\int_\U g(x) : X(x)^2 : \dd x}$$ 
and \eqref{zidane} implies that the sequence is bounded in $\cL^2$ and hence tight in $\cL^1$.
Thus the convergence in probability implies convergence in $\cL^1$ and thus convergence of the expectation which concludes the proof of Lemma \ref{waddle}.

\medskip

So we just have to prove  \eqref{zidane}. For notational simplicity we assume in the proof that $g \equiv 1/2$ and that $\int_\U \dd x=1$ 
but the proof with general $g$ (and in particular $g=e^{U(x)}$) works just the same. As the quantity are continuous for $\gamma \in (0,1]$, we can assume 
that $\gamma$ is small.

Our strategy is to introduce first the white noise cutoff $(X_\gep)_\gep$  for the field $X$. 

\vspace{1mm}

{\bf Introducing the cutoff.} Recall that $:e^{\gamma X_\gep(x)}:$ stands for $e^{\gamma X_\gep(x)-\frac{\gamma^2}{2}\E[X_\gep(x)^2]}$. Then as $e^x-x-1$ is a positive function, we have
\begin{equation}\label{iniesta}
-\gamma^{-2} \left(C(x,\U)^{\frac{\gamma^2}{2}}:e^{\gamma X_\gep(x)}:- 1 -\gamma X_\gep(x)-\frac{\gamma^2}{2}\ln C(x,\U)\right)\le \frac{1}{2}\E\left[X^2_{\gep}(x)\right]\leq\frac{1}{2}|\log \gep|.
\end{equation}
For the rest of the proof, we use the notation

\begin{equation}\label{kabadiawara}
R_\gep(x):= C(x,\U)^{\frac{\gamma^2}{2}} :e^{\gamma X_\gep(x)}:- 1 -\gamma X_\gep(x)-\frac{\gamma^2}{2}\ln C(x,\U)-\frac{\gamma^2}{2} :X^2_\gep(x):
\end{equation}

We want to find a good bound for $\int_\U | R_\gep(x)| \dd x$ which holds with a large probability and use
\eqref{iniesta} to bound the exponential when the bound on $\int_\U | R_\gep(x) | \dd x$ is not satisfactory. We set $\gep:=e^{-\gamma^{-1/8}}$ and hence $|\log \gep| = \gamma^{-1/8}$. We stress that this relation will hold during the rest of the proof of lemma \ref{zidane}.

We have for any $x$
\begin{equation}\label{xavi}
\E \left[|  R_\gep(x)| \ind_{\{|X_{\gep}(x)|>|\log \gep|^2\}}\right]
\le e^{-c | \log \gep |^3}(1+|\ln C(x,\U)|).
\end{equation}
This inequality can be established without being subtle: use the triangle inequality to decompose $| R_\gep(x)|$ and then estimate each term with standard Gaussian computations.

Separating the space into the two events $\{|X_{\gep}(x)|>|\log \gep|^2\} $ and $\{|X_{\gep}(x)|\leq |\log \gep|^2\}$ and using the inequality    $e^u-1-u-u^2/2 \geq  u^3/6$ for $u\in\R$ on the second event, we deduce
\begin{equation}\label{ribery}
R_\gep(x)
\ge   R_\gep(x)   \ind_{\{|X_{\gep}(x)|>|\log \gep|^2\}}-C \gamma^{9/4}(1+|\ln C(x,\U)|^3).
\end{equation}
Let us set
$$\cA:=\left\{ \left(\int_{\U}      | R_\gep(x) |   \ind_{\{|X_{\gep}(x)|>|\log \gep|^2\}}\,\dd x\right)  \ge \gamma^3\right\}.$$

We can integrate the   inequality \eqref{ribery} with respect to the Lebesgue measure over $\U$. To bound the integrated first term in the right-hand side of \eqref{ribery}, we can use \eqref{xavi} and the Markov inequality to obtain
\begin{equation}\label{benzema}
\P\left[ \cA \right]\le 
e^{-c | \log \gep |^3}\gamma^{-3}.
\end{equation}
We can finally conclude, using \eqref{iniesta} and \eqref{ribery}, that
\begin{multline}
\mathbb E\left[ \exp\left(-\frac{1}{\gamma^{2}}\int_\U (C(x,\U)^{\frac{\gamma^2}{2}}:e^{\gamma X_\gep(x)}:- 1 -\gamma X_\gep(x)-\frac{\gamma^2}{2}\ln C(x,\U))\,\dd x\right) \right]
\\ \le e^{C\gamma^{1/4}}\E\left[e^{-\frac 1 2\int_{\U} : X^2_\gep:} \ind_{\cA^c}\right] +\P\left[\cA\right] e^{\frac{1}{2}|\log \gep|},
\end{multline}
which can be bounded above by a constant (independent of $\gamma$) thanks to  the bound $$\sup_{\gep>0}\E\left[e^{-\frac {1} {2} \int_{\U} : X^2_\gep:}\right]<+\infty$$ and \eqref{benzema}.
        
\vspace{2mm}

{\bf Removing the cutoff.} The first observation is that we have the following estimate for any event $\mathcal B$ by using the Cauchy-Schwarz inequality (and the fact that the exponential is positive and $C(x,\U)\leq 1$)
\begin{equation}\label{zlatan}
\mathbb E\left[ e^{-\frac{1}{\gamma^{2}} \int_\U (e^{\gamma X(x)}- 1 -\gamma X(x)-\frac{\gamma^2}{2}\ln C(x,\U))\,\dd x}\ind_{\cB} \right]\le 
e^{\gamma^{-2}}\mathbb E\left[ e^{\frac{1}{\gamma}\int_\U X(x)\,\dd x}\ind_{\cB} \right]\le e^{C\gamma^{-2}}\sqrt{\P\left[\cB\right]},
\end{equation}
where $C$ is a positive constant.
Now we set $M=\int_\U :\!e^{\gamma X(x)}\!\!:\dd x$, $M_\gep= \int_\U C(x,\U))^{\gamma^2/2} :\!e^{\gamma X_\gep(x)}\!\!:\dd x$. On the event $\cB^c$, we can write
\begin{align}\label{cavani}
\mathbb E\Big[  &e^{-\frac{1}{\gamma^{2}} \int_\U (e^{\gamma X(x)}- 1 -\gamma X(x)-\frac{\gamma^2}{2}\ln C(x,\U))\,\dd x}\ind_{\cB^c}\Big]\nonumber\\
= &\mathbb E\left[ e^{-\frac{1}{\gamma^{2}} \int_\U (C(x,\U))^{\gamma^2/2}:e^{\gamma X_\gep(x)}:- 1 -\gamma X_\gep(x)-\frac{\gamma^2}{2}\ln C(x,\U))\,\dd x}e^{\frac{1}{\gamma}\int_\U  (X-X_\gep)(x)\,\dd x}e^{\frac{1}{\gamma^{2}}(M_\gep-M)}\ind_{\cB^c} \right].
\end{align}
It is therefore relevant to consider an event $\cB$ such that we can properly estimate the last two exponential terms. A reasonable choice is to set
$$\cB:=\Big\{| \int_\U(X-X_\gep)(x)\dd x | \ge \gamma^2\Big\} \cup \Big\{(M_\gep-M)\ge \gamma^3\Big\}.$$ 
On the event $\cB^c$, \eqref{cavani} allows us to compare (with constants) the desired quantity with the cutoff version.

To   use \eqref{zlatan} on the event $\cB$, what remains to do is to prove that 
\begin{equation}
\P(\cB)\le e^{-3C\gamma^{-2}}  .
\end{equation}

The quantity $\int_\U(X-X_\gep)(x)\dd x$ is a Gaussian random variable whose variance is of order $\gep$.
Hence there exists a constant $c$ such that
\begin{equation}\label{llacer}
\P \left[ \int_\U(X-X_\gep)(x)\dd x\ge  \gamma^2 \right] \le \exp(-c \gamma^4 \gep^{-1})
\end{equation}
To evaluate the likeliness of a deviation of $M_\gep-M$, we are going to compute the exponential moment of this variable with respect to $\E^\gep:=\E\left[  \cdot | \cF_\gep\right]$ 
where $\mathcal{F}_\gep$ is the sigma-algebra generated by the random variables $\{X_{u}(x);\gep\leq u,x\in\U\}$.
For $t>0$ let us consider the function
\begin{equation}
\phi(t)=\E\left[ e^{t(M_\gep-M)} \ | \ \cF_\gep \right]<\infty.
\end{equation}
Note that neither the full expectation with respect to $\P$ nor the  expectation for negative $t$ are finite.
We have
\begin{equation}\label{fiprim}
\phi'(t)= \E^\gep\left[ (M_\gep-M) e^{t(M_\gep-M)}  \right]\\
=\int_\U C(x,\U)^{\gamma^2/2} :e^{\gamma X_{\gep}(x)}:\E^\gep\left[  e^{t(M_\gep-M)}-e^{t(M_\gep-\tilde M^x)} \right] \dd x
\end{equation}
where 
\begin{equation}\label{tildem}
\tilde M^x:=\int_\U e^{\gamma X(y)} e^{\gamma^2 \bar G_\gep(x,y)} \dd y,
\end{equation}
and $\bar G_\gep(x,y)>0$ is the correlation function of $X-X_\gep$.
Note that for any $x\in \U$
$$e^{t(M-\tilde M^x)}=\exp\left( -t \int_\U (e^{\gamma^2 \bar G_\gep(x,y)}-1)e^{\gamma X(y)}\dd y \right) $$ and $e^{-tM}$ are decreasing functions of the field $X-X_\gep$.
Hence making use of the FKG inequality for white noise (see \cite[section 2.2]{grimmett} for the case of countable product and note that by expression \eqref{GFFXY} the field $X-X_\epsilon$ is an increasing function of the white noise) for the field $X-X_\gep$ and the inequality $e^u\geq 1+u$ 
\begin{align*}
\E^\gep\left[ e^{-t\tilde M^x }\right]\ge &\E^\gep\left[ e^{t(M-\tilde M^x)} \right] \E^\gep \left[ e^{-tM} \right]\\
\ge &\E^\gep\left[ 1+ t(M-\tilde M^x)\right]\E^\gep \left[ e^{-tM} \right]\\
=&\Big(1- t \int_\U \left[ e^{\gamma^2 \bar G_\gep(x,y)}-1 \right] C(x,\U)^{\gamma^2/2} :e^{\gamma X_{\gep}(y)}:\dd y\Big)\E^\gep \left[ e^{-tM} \right].
\end{align*}
Combining this with \eqref{fiprim} and using that $C(x,\U))$ is uniformly bounded  and   $:e^{\gamma X_{\gep}(y)}: \le e^{\gamma X_{\gep}}$, we obtain that 
for $\gamma$ sufficiently small 
\begin{equation}
\phi'(t)\le 2  t \phi(t)\iint_{\U^2} \left[e^{\gamma^2 \bar G_\gep(x,y)}-1\right] e^{\gamma X_{\gep}(x)+X_\gep(y)} \dd x\dd y.
\end{equation}
With this we can conclude that 
\begin{equation}\label{cbeau}
\phi(t)\le e^{Z_\gep t^2},
\end{equation}
where
$$Z_\gep:=\iint_{\U^2} \left[e^{\gamma^2 \bar G_\gep(x,y)}-1\right]e^{\gamma X_{\gep}(x)+\gamma X_\gep(y)}  \dd x\dd y.$$

The last thing that we need is a good control on $Z_\gep$.
Note that one can find a constant $C$ such that for all $x\in \U$, for all $\gamma<1$, and $\gep$
\begin{equation}
 \int_{\U} \left[e^{\gamma^2 \bar G_\gep(x,y)}-1\right]\dd y\le  C \gamma^2 \gep.
\end{equation}
Hence using the inequality 
\begin{equation}
e^{\gamma X_{\gep}(x)+X_\gep(y)}\le \frac{1}{2}\left(e^{2\gamma X_{\gep}(x)}+e^{2\gamma X_\gep(y)}\right),
\end{equation}
and symmetries in the integration we obtain
\begin{equation}
Z_\gep\le \iint_{\U^2} \left[e^{\gamma^2 \bar G_\gep(x,y)}-1\right]  e^{2\gamma X_\gep(x)} \dd y\dd x
\le C \gamma^2 \gep  \int_\U e^{2\gamma X_\gep(x) } \dd x.
\end{equation}
Now we have
\begin{equation}
 \int_\U e^{2\gamma X_\gep(x)} \dd x= \gep^{-1/2}+ \int_\U e^{2\gamma X_\gep(x)} \ind_{\{\gamma X_\gep(x)\ge   | \log \gep |/4\}} \dd x
\end{equation}
and one can find $c>0$ such that
\begin{equation}
\E\left[ \int_\U e^{2\gamma X_\gep(x)} \ind_{\{\gamma X_\gep(x)\ge   | \log \gep |/4\}} \dd x\right] \le e^{-c\gamma^{-2} |\log \gep|}.
\end{equation}
Hence using the Markov property (changing the value of $c$  if needed) we have for $\gamma$ small enough

\begin{equation}
\P\left[Z_\gep \ge \sqrt{\gep} \right]\le  e^{-c\gamma^{-2} |\log \gep|}.
\end{equation}
Finally using \eqref{cbeau} with $t=\gep^{-1/4}$ we have that 

\begin{equation*}
\P\left[ (M_\gep-M)\ge \gamma^3 \ | \ Z_\gep \le \gep^{1/2}\right]\le \exp(-\gep^{-1/4}\gamma^3)e^{1/2}.
\end{equation*}
This entails that
\begin{equation*}
\P\left[ (M_\gep-M)\ge \gamma^3  \right]\le e^{-c\gamma^{-2} |\log \gep|}+  2 \exp(-\gep^{-1/4}\gamma^3)e^{1/2}.
\end{equation*}
which completes the proof since we have $\gep =e^{-\gamma^{-1/8}}$.\qed

\bigskip

Back to the proof of Theorem \ref{th:semi}, we  consider a continuous bounded function $F$ on the space  $H^{-1}(\U)$. The same computation as \eqref{vampetta} shows that
\begin{align}\label{Messi}
  \E\big[F(\gamma X)\exp\big(- \frac{4\pi\Lambda}{\gamma^2} \int_\U e^{\gamma X(x)} \, \dd x \big)\big]   &=\exp\left(-  \frac{ 1}{4\pi \gamma^2} \int_\U ( |\partial U(x)|^2+16\pi^2\Lambda e^{ U(x)}) \dd x\right) \\
  &\times\E\Big[ F(\gamma X +U)  \exp\left(- \frac{4\pi\Lambda}{\gamma^2} \int_\U e^{ U(x)} (e^{\gamma X(x)} -1- \gamma X(x)) \dd x \right)\Big].\nonumber
\end{align} 
As $\gamma X$ converges in probability towards $0$, it is plain to deduce from Lemma \ref{waddle} again that
the expectation in the above right-hand side behaves when $\gamma\to 0$ as follows
\begin{multline}\label{Messi2}
\lim_{\gamma\to 0}\E\Big[ F(\gamma X +U)  \exp\left(- \frac{4\pi\Lambda}{\gamma^2} \int_\U e^{ U(x)} (e^{\gamma X(x)} -1- \gamma X(x)) \dd x \right)\Big]
\\ = F(U) e^{- 2\pi\Lambda \int_{\U}e^{ U(x)}  \log C(x,\U) \dd x} \E[  \exp\big( -2\pi\Lambda\int_\U  e^{U(x)}  : X(x)^2 : \dd x   \big)   ].
\end{multline} 
This shows the convergence in law of the field $\gamma X$ towards $U$. Hence item 3. Item 4 can be proved in the same way.

Now we focus on item 5. We use again the notation $$Y=X-\gamma^{-1}U.$$  From \eqref{Messi}, we have
\begin{equation}
\lim_{\gamma\to 0}\E_{\mu,\gamma}[F(Y)]=Z^{-1}\E[F(X)\exp\big( -2\pi\Lambda\int_\U  e^{U(x)}  : X(x)^2 : \dd x \big)  \big]
\end{equation}
 where 
 
 \begin{equation}
 Z:=\E[\exp\big(-2\pi\Lambda\int_\U  e^{U(x)}  : X(x)^2 : \dd x \big) ].
 \end{equation} 
Using Lemma \ref{lem:neymar}, the proof of Theorem \ref{th:semi} is over.\qed

\vspace{2mm}

\subsection{Proof of the large deviation principle}

\subsubsection*{Proof of Proposition \ref{ratefunction}} The fact that $I^*$ is a good rate function will follow from Theorem \ref{th:ldp}. So we focus on establishing 
the expression of $I^*$ on $H^1_0(\U)$. Our strategy is to first establish the identity on a dense subset of $H^1_0(\U)$ and then use a bit of topology to extend it.

 For $h\in H^{-1}(\U)$, let us denote by $H$ the mapping
$$H:f\in H^1_0(\U)\mapsto h(f)-(\tf(\Lambda,f)-\tf(\Lambda)).$$
As it is G\^ateaux-differentiable (see Proposition \ref{th:liouville2}), we can compute the partial derivative evaluated at $f$ in the direction $v\in H^1_0(\U)$, call it $\partial_vH(f)$. To this purpose, let us introduce the solution $V$ of \eqref{boucherie} and the function
$g\in H^1_0(\U)$ solution of \eqref{vanbasten} with $h$ in \eqref{vanbasten} replaced by $v$.
From Proposition \ref{th:liouville2} and the definition of $\tf$.
\begin{equation}\begin{split}
\partial_vH(f)&=\lim_{t\to \infty} \frac{H(f+tv)-H(f)}{t}\\
=& h(v)+\frac{1}{4\pi}\int_\U \big(2\langle \partial V,\partial g\rangle +16 \pi^2\Lambda e^{V}g-4\pi f g \big)\dd x-\int_\U v(V -U )\,\dd x\\
=&h(v) -\int_\U v(V -U )\,\dd x.\label{string}
\end{split}\end{equation}
To get the last line, we have used the fact that $V$ is the solution of \eqref{boucherie} in such a way that the first integral in the first line vanishes. Let us define
\begin{equation}\label{defs}
\mathcal{S}=\{h\in H^1_0(\U); h=V-U;  V\text{ solution to \eqref{boucherie} for some function }f\in H^1_0(\U)\}.
\end{equation}
If $h\in \mathcal{S}$, we can choose $f$ such that $h=V-U$ where $V$ is a solution  to \eqref{boucherie}. Then, by \eqref{string}, we deduce that $f$ is a critical point of $H$, which is concave. Therefore $I^*(h)=H(f)$. Plugging the relation $h=V-U$ into the expression of $H(f)$, we get $I^*(h)=  E(V)-E(U)=E(U+h)-E(U)$. This provides the expression of $I^*$ on $\mathcal{S}$.
 
 \medskip
 
Now we show that $\mathcal{S}$ is  dense in $H^1_0(\U)$.
For this purpose we show that 
\begin{equation}\label{dqsfhjytrdfgbvnjuytfg}
\mathcal{S}=\{h\in H^{1}_0(\U) \  | \  \Delta h \in H^1_0(\U) \}.
\end{equation}
Indeed setting $V=U+h\in H^1_0(\U)$, we have
\begin{align*}
\triangle V=&\triangle U+\triangle h\\
=&8\pi^2 \Lambda e^V +(\triangle U+\triangle h-8\pi^2 \Lambda e^V)\\
=&8\pi^2 \Lambda e^V +(8\pi^2 \Lambda e^U+\triangle h-8\pi^2 \Lambda e^V).
\end{align*} 
Setting $f= (4\pi \Lambda e^V- 4\pi \Lambda e^U-(2\pi)^{-1}\triangle h)$, it remains to 
 to prove that $f$ belongs to $H^1_0(\U)$ if and only if $\Delta h$ does.
In both cases it suffices to prove that $e^U-e^V \in H^1_0(\U)$ but this is easy:
 because it vanishes on the boundary because $h$ does, and
  $\partial (e^U-e^V)= e^U \partial U  -e^V\partial V$ is square integrable.
For this last point, one can check that either $\Delta h\in H^1_0(\U)$ (easy) or $f\in H^1_0(\U)$ (using 
$\gD V= 8\pi^2 \Lambda e^V-2\pi f$) implies that $V$ is bounded.

\medskip

Now we establish the expression of $I^*$ on $H^1_0(\U)\setminus \mathcal S$ by density.
Note that  as a supremum of continuous linear functions, $I^*$ restricted to $H^1_0(\U)$ is weakly lower semi-continuous (for the $H^1_0(\U)$ norm). 
By continuity of $E(.)$ for the $H^1_0(\U)$ norm, approximating $h\in H^1_0$ by a 
sequence in $\mathcal S$ we deduce that 
\begin{equation}
\forall h\in H^1_0(\U), \quad I^*(h)\leq E(U+h)-E(U) .
\end{equation}
 Conversely, if $h\in H^1_0(\U)$, we can find a sequence $(h_n)_n$ in  $\mathcal{S}$ converging in $H^1_0(\U)$ towards $h$. For each $n$, let us consider $f_n\in H^1_0(\U)$ such that $h_n=V_n-U$ where $V_n$ is the solution to \eqref{boucherie} associated to $f_n$. From \eqref{boucherie}, one can  check that 
 $(f_n)_n$ is Cauchy and  strongly converges in $H^{-1}(\U)$. Then we get:
\begin{align*}
I^*(h)=&\sup_{f\in H^1_0(\U)} h(f)-\tf(\Lambda,f)+\tf(\Lambda)\\
\geq & h(f_n)-\tf(\Lambda,f_n)+\tf(\Lambda)\\
=& (h-h_n)(f_n)+h_n(f_n)-\tf(\Lambda,f_n)+\tf(\Lambda)\\
=& (h-h_n)(f_n)+E(U+h_n)-E(U).
\end{align*}
We conclude by observing that $(h-h_n)(f_n)\to 0$ and $E(U+h_n)-E(U)\to E(U+h)-E(U)$ as $n\to\infty$.

\medskip
To complete the proof, 
we show that $I^*(h)<+\infty$ implies $h\in H^1_0(\U)$. For each $\bar{f}\in H^1_0(\U)$, let us consider the associated solution to \eqref{boucherie} and define $\bar{h}=\bar{V}-U$. Repeating the above argument, we have for each $\bar{f}\in H^1_0(\U)$
\begin{equation}
I^*(h)\geq (h-\bar{h})(\bar{f})+E(U+\bar{h})-E(U).
\end{equation}
As $I^*(h)<+\infty$ and $E(U+\bar{h})-E(U)\geq 0$, we deduce that
\begin{equation}\label{zola}
\forall \bar{f}\in H^1_0(\U),\quad h(\bar{f})\leq C+\bar{h}(\bar{f}),
\end{equation}
for some constant $C>0$, which does not depend on $\bar{f}$. Let us further introduce a function $\bar{g}\in H^1_0(\U)$ such that $-2\pi \bar{f}=\triangle \bar{g}$. To establish that $h \in H^1_0(\U)$, it suffices to prove that the above right-hand side of \eqref{zola} is bounded uniformly when
$\|\bar{f}\|_{H^{-1}}\leq 1$ (or equivalently $\|\bar{g}\|_{H^{1}}\leq 1$). By integrating \eqref{boucherie} with respect to $\bar{V}$, we get
$$-\int_\U |\partial \bar{V}|^2\,\dd x=8\pi^2\Lambda \int e^{\bar{V}}\bar{V}\,\dd x+\int_\U \triangle \bar{g}\bar{V}\,\dd x,$$
which can be rewritten as
\begin{equation}
\int_\U |\partial \bar{V}|^2\,\dd x+8\pi^2\Lambda\int_\U (e^{\bar{V}}-1)\bar{V}\,\dd x=-8\pi^2\Lambda\int_\U \bar{V}\,\dd x +\int_\U \langle\partial \bar{g},\partial \bar{V}\rangle\,\dd x.
\end{equation}
Using the elementary inequality $\langle a,b\rangle \leq \frac{1}{2 c}|a|^2+\frac{c}{2  }|b|^2$ for the two terms  for a well chosen $c>0$ we can establish  that the right-hand side is less than $C'+\frac{1}{2}\int_\U |\partial \bar{V}|^2\,\dd x$, for some constant $C'$ that does not depend on $\bar{f}$. Observe that $(e^u-1)u\geq 0$ for all $u\in\R$ so we deduce that $\int_\U |\partial \bar{V}|^2\,\dd x$  is bounded uniformly on the set $\{ \bar{f}\in H^1_0(\U) \ | \ \|\bar{f}\|_{H^{-1}} \leq 1\}$ (and thus $\int_\U |\partial \bar{h}|^2\,\dd x$ too). Finally, 
we have 
$$ \bar{h}(\bar{f})= \int_\U \langle\partial \bar{g},\partial \bar{h}\rangle\,\dd x\leq |\bar{g}|_{H^1}|\bar{h}|_{H^1}$$
so that $ \bar{h}(\bar{f})$ is uniformly bounded on the set $\{ \bar{f}\in H^1_0(\U) \ | \ \|\bar{f}\|_{H^{-1}} \leq 1 \}$. This implies that $h\in H^1_0(\U)$.

Also, recall that $U$ is the unique minimum in $H^1_0(\U)$ of the functional $E$. Indeed, a function in $H^1_0(\U)$ is a minimum of this functional if and only if it is a weak solution to \eqref{eq:class}. Furthermore, the weak solution of \eqref{eq:class} is unique as we have proved that the field $\gamma X$ converges in law (and even in probability) towards $U$ as soon as we get a weak solution $U$ to this equation. The limit in law being unique, we get uniqueness for \eqref{eq:class}. In particular, if $h\in H^1_0(\U)$ and $h\not = 0$, we get $I^*(h)=E(U+h)-E(U)>0$. \qed

 
\medskip

\subsubsection*{ Proof of Theorem \ref{th:ldp}.} Assume that we can prove that the family $(Y_\gamma)_\gamma$ is exponentially tight and that for each function $f\in H^1_0(\U)$
\begin{equation}\label{expLDP}
\lim_{\gamma\to 0}\gamma^2\ln \E_{\mu,\gamma}\big[e^{ \frac{Y_\gamma(f)}{\gamma^2}}\big]=\tf(\Lambda,f)-\tf(\gL).
\end{equation}
The mapping $f \in H^1_0(\U)\mapsto \tf(\Lambda,f)-\tf(\gL)$ is G\^ateaux-differentiable as shown in \eqref{string} and weakly lower semi-continuous (even weakly continuous) from  Proposition \ref{th:liouville3}. Hence we can apply a standard result from the theory of Large deviation in functional spaces
\cite[Corollary 4.5.27]{dembo}, which entails  the proof of Theorem \ref{th:ldp}.

\medskip

So we focus on establishing \eqref{expLDP} first and then we will prove that the family $(Y_\gamma)_\gamma$ is exponentially tight.
As we already know the asymptotic behavior of the partition function, it is sufficient to compute the asymptotic behavior of 

$$Z_{\mu,\gamma}[e^{ \frac{Y_\gamma(f)}{\gamma^2}}\big]= \E\big[\exp\left( \frac{Y_\gamma(f)}{\gamma^2}- \frac{4\pi\Lambda}{\gamma^2} \int_\U e^{\gamma X(x)} \, \dd x \right) \big].$$
Let  
$V$ be the (deterministic) weak solution of \eqref{boucherie} and 
set 
\begin{equation}\label{defthetabis}
\theta(x):=  f(x) -4\pi\Lambda e^{V (x)}=-\frac{\Delta V(x)}{2\pi}.
\end{equation}
Note that it implies
\begin{equation}\label{christanval}
\int_\U  \theta(y) G_\U(x,y)\dd y=V(x)
\end{equation}
We define $H_\gamma$ to be a shifted version of the field $X$,
\begin{equation}
H_\gamma(x)=X-\frac{V}{\gamma}.
\end{equation}
We have
\begin{align}
 & \E\big[\exp\left( \frac{Y_\gamma(f)}{\gamma^2}- \frac{4\pi\Lambda}{\gamma^2} \int_\U e^{\gamma X(x)} \, \dd x \right) \big] \\
  &=\E\big[e^{\frac{1}{\gamma^2} \int_{\U}(\gamma X(x)-U(x))f(x)\dd x -\frac{4\pi\Lambda}{\gamma^2} \int_{\U} e^{V(x)}(1+\gamma H_\gamma(x)) \, \dd x }e^{- \frac{4\pi\Lambda}{\gamma^2} \int_\U e^{\gamma X(x)}-e^{V(x)}(1+\gamma H_\gamma(x)) \, \dd x }\big]\nonumber  \\
 &= e^{   \frac{1}{\gamma^2} \int_\U (4\pi\Lambda V(x)e^{V(x)}-4\pi\Lambda e^{V(x)}-U(x)f(x)) 
 \dd x+ \frac{1}{2 \gamma^2}   \iint_{\U^2}  \theta(x)\theta(y)  G_\U(x,y)\dd x \dd y }\nonumber\\
 &\times
 \E\big[ e^{\frac{1}{\gamma } \int_\U   X(x) \theta(x)\dd x- \frac{1}{2 \gamma^2}   \int_{\U^2}  \theta(x)\theta(y)  G_\U(x,y)\dd x \dd y} 
 e^{- \frac{4\pi\Lambda}{\gamma^2} \int_\U e^{\gamma X(x)}-e^{V(x)}(1+\gamma H_\gamma(x)) \, \dd x }\big] .\label{valbuenatackle}
\end{align}
Once again, the first exponential term in the expectation 
$$e^{\frac{1}{\gamma } \int_\U   X(x) \theta(x)\dd x- \frac{1}{2 \gamma^2}   \int_{\U^2}  \theta(x)\theta(y)  G_\U(x,y)\dd x \dd y} 
 $$
is a  Cameron-Martin transform term. It has the effect of shifting the field $X$ by an amount $\gamma^{-1}V$ (cf. \eqref{christanval}),
and hence after this shift, $H_\gamma$ becomes a centered field, and the expectation in the last line of \eqref{valbuenatackle} is equal to 
\begin{align*}
\E\Big[   \exp\left(- \frac{4\pi\Lambda}{\gamma^2} \int_\U e^{ V(x)} (e^{\gamma X(x)}-1- \gamma X(x) )\dd x \right)\Big] .
\end{align*} 
Concerning the exponential term in front of the expectation, it can be simplified. Let us briefly explain how. From \eqref{christanval} and \eqref{defthetabis}
\begin{equation}\label{laudrup}
 \iint_{\U^2}  \theta(x)\theta(y)G_\U(x,y)\dd x \dd y= \int_{\U}  \theta(x)V(x) \dd x=
 -\frac{1}{2\pi}\int_{\U} \Delta V(x) V(x) \dd x= \frac{1}{2\pi}\int_{\U}  |\partial V(x)|^2 \dd x.
\end{equation}
We also have 
\begin{equation}\label{toefting}
 \int_{\U} 4\pi \Lambda e^{V(x)} V(x) \dd x= \int_{\U}  (f(x)-\theta(x))V(x)=-\frac{1}{2\pi}\int_{\U}  |\partial V(x)|^2 \dd x+  \int_{\U} f(x)V(x).
\end{equation}
Using this in \eqref{valbuenatackle} we obtain 
\begin{equation}\begin{split}
 & \E\big[\exp\left( \frac{Y_\gamma(f)}{\gamma^2}- \frac{4\pi\Lambda}{\gamma^2} \int_\U e^{\gamma X(x)} \, \dd x \right) \big] \\
&\quad \quad =\exp\left(- \frac{ 1}{4\pi \gamma^2} \int_\U ( |\partial V(x)|^2+16\pi^2\Lambda e^{ V(x)}) +\frac{1}{\gamma^2}\int_\U f(x)( V(x)-U(x))\dd x\right)\\
&\quad \quad \quad \quad \times \E\Big[   \exp\left(- \frac{4\pi\Lambda}{\gamma^2} \int_\U e^{ V(x)} (e^{\gamma X(x)}-1- \gamma X(x) )\dd x \right)\Big].
\end{split}\end{equation}
To complete the proof of  \eqref{expLDP}: we use Lemma \ref{waddle} which asserts that  the last line converges as $\gamma\to 0$ towards 
$$e^{-  2\pi\Lambda  \int_\U e^{ V(x)}\ln C(x,\U) \dd x }\E[e^{-2\pi\Lambda\int_\U e^{V(x)}:X^2(x):\dd x}].$$

Now, we turn to the exponential tightness of the field $Y_\gamma=\gamma X-U$. The exponential tightness of the field $Y_\gamma=\gamma X-U$ is equivalent to the exponential tightness of $\gamma X$ 
(simply because if $K$ is a compact set of $H^{-1}(\U)$, 
$K+U$ is also compact).
We adopt the framework of section 4.2 in \cite{dubedat}. By conformal invariance, we work on the square $S=[0,1]^2$. In this case, given a sequence $(a_{j,k})_{j,k \geq 1}$, the series  
\begin{equation}\label{defserie}
f_n:=\sum_{1 \leq j,k \leq n}  a_{j,k} \sin(\pi j x) \sin( \pi j y )
\end{equation}
converges in $H^{-1}(S)$ if and only if $\sum_{j,k \geq 1}  \frac{|a_{j,k}|^2}{j^2+k^2}< \infty $. In this case the limit $f:= \lim_n f_n$ has the following norm
\begin{equation*}
|f|_{H^{-1}(S)}= \sum_{j,k \geq 1}  \frac{|a_{j,k}|^2}{(j^2+k^2)}.
\end{equation*}
In  $H^{-1}(S)$, the GFF is then the almost sure limit of the series \eqref{defserie} where $a_{j,k}= \frac{\epsilon_{j,k}}{\sqrt{j^2+k^2}}$
where $(\epsilon_{j,k})_{j,k \geq 1}$ is an i.i.d. sequence of standard Gaussian variables (in this case the $\epsilon_{j,k}$ are the $H^1(S)$
projections of $X$ on the $H^1(S)$ basis $((x,y) \rightarrow \sin(\pi j x) \sin( \pi j y ))_{j,k \geq 1}$). Let $C>0$ be fixed. We introduce the
following compact set of $H^{-1}(S)$ (we identify the limit of the series \eqref{defserie} with the sequence $(a_{j,k})_{j,k \geq 1}$) 
\begin{equation*}
K_C = \lbrace  (a_{j,k})_{j,k \geq 1}; \:    \forall j,k, \; |a_{j,k}| \leq \frac{C}{(j^2+k^2)^{1/4}}  \rbrace 
\end{equation*}
We have 
\begin{align*}
\P( \gamma X \notin K_C) &= \P(  \exists j,k, \; \gamma |\epsilon_{j,k}| > C (j^2+k^2)^{1/4}  ) \\
& \leq \sum_{j,k \geq 1}  \P( \gamma |\epsilon_{j,k}| > C (j^2+k^2)^{1/4}  )  \\
& \leq \sum_{j,k \geq 1}  e^{-\frac{C^2 ((j^2+k^2)^{1/2})}{2 \gamma^2} }. 
\end{align*}  
hence we get that 
\begin{equation*}
\underset{C \to \infty}{\lim}  \underset{\gamma \to 0}{\overline{\lim}} \gamma^2 \log \P( \gamma X \notin K_C) = - \infty.
\end{equation*}
This shows that $\gamma X$ is exponentially tight in $H^{-1}(S)$.

\qed


 
\section{Semiclassical limit of LFT with heavy matter insertions}\label{sec:heavy}
In this section, we want to treat the case of  heavy matter  operator insertions in the partition function. This roughly corresponds to tilting the partition function of LFT with exponential terms and we will see that, semiclassically, this creates conical singularities in a hyperbolic surface. We restrict once again to the flat unit disk for simplicity.

More precisely, we consider distinct $z_1,\dots,z_p\in\U$, $\cX_1,\dots \cX_p\in [0,2[$, a cosmological constant $\mu\geq 0$ and a Liouville conformal factor $\gamma\in]0,2]$.  We set 
$$Q=\frac{2}{\gamma}+\frac{\gamma}{2}.$$ 
We formally define the law $\P_{\mu,\gamma,(z_i,\cX_i)_i}$ of the Liouville field  $X$ on $\U$ with heavy matter insertions $(z_i,\cX_i)_i$ associated to $(\mu,\gamma)$ as the law of the GFF on $\U$ tilted by
\begin{equation}\label{formaldef}
\exp\Big(-4\pi \mu\int_{{\U}} \!e^{\gamma X(x)}  \dd x\Big)\prod_{i=1}^p e^{\frac{\cX_i}{\gamma}X(z_i)},
\end{equation}
Of course, the above expression is not a function (because of $e^{\frac{\cX_i}{\gamma}X(z_i)}$) 
and this cannot be considered as a Radon-Nykodym derivative, but on a formal level one can always consider this last term as a Cameron-Martin tilt.
The rigorous definition of $\P_{\mu,\gamma,(z_i,\cX_i)_i}$ is then given by its action on  bounded continuous functionals $F$ on  $H^{-1}(\U)$   
as follows
\begin{align*}
&\E_{\mu,\gamma,(z_i,\cX_i)_i}[F(X)]\\
&=Z^{-1}_{\mu,\gamma,(z_i,\cX_i)_i} \E\Big[F\big(X+\sum_i \frac{\cX_i}{\gamma}G_\U(\cdot,z_i)\big)\exp\Big(-4\pi \mu\int_{{\U}}e^{\gamma X(x)} e^{ \sum_i \cX_iG_\U(\cdot,z_i)}\dd x \Big) \Big]
\end{align*}
where  $\E$ stands for the expectation with respect to the free field $X$, and
\begin{equation}
 Z_{\mu,\gamma,(z_i,\cX_i)_i}= \E\Big[\exp\Big(-4\pi \mu\int_{{\U}} e^{\gamma X(x)}  e^{ \sum_i \cX_iG_\U(\cdot,z_i)}\dd x \Big) \Big].
\end{equation}

 The additional exponential terms in the above product are called heavy matter operators in the physics literature (see \cite{nakayama,witten} for instance). The problem is to compute the asymptotic behaviour of the partition function and to find the limit in law  under the probability law  $\P_{\mu,\gamma,(z_i,\cX_i)_i}$ of the field $\gamma X$ when
$\gamma^2\mu=\Lambda$ and $\gamma\to 0$.

%
%
%
We will see that the field concentrates on the solutions of the Liouville equation with sources (see Theorem \ref{th:liouvillesource})
\begin{equation}\label{eq:source}
\triangle U =8\pi^2\Lambda e^U- 2\pi \sum_i \cX_i\delta_{z_i}\quad \quad U_{|\partial\U}=0,
\end{equation}
 where $\delta_z$ stands for the Dirac mass at $z$.  Theorem \ref{th:liouvillesource} shows that if $U$ is the solution of equation \eqref{eq:source} then $U-\sum_{i}\cX_i G_\U(\cdot,z_i)$ is at least continuous. Therefore the metric $e^{U(x)}dx^2$ possesses singularities of the type $\frac{1}{|x-z_i|^{\cX_i}}$ at the points $z_i$ .

\begin{theorem}\label{th:source} 
Assume $\gamma\to 0$ while keeping fixed the quantity $\gamma^2\mu=\Lambda$. The field  $\gamma X$  concentrates on the solution of the  Liouville equation with sources \eqref{eq:source}.  More precisely
\begin{enumerate}
\item The partition function has the following asymptotic behavior at the exponential scale 
\begin{equation}
\lim_{\gamma\to 0} \gamma^2 \log Z_{\mu,\gamma,(z_i,\cX_i)_i}=  -    \frac{1}{4\pi\gamma^2} \int_\U(|\partial (U-H)(x)|^2+16\pi^2\Lambda e^{U(x)} )  \dd x=: \tf(\gL, (z_i,\cX_i)_i).
\end{equation}
where  $H(x)= \sum_i \cX_iG_\U(\cdot,z_i)$.
\item More precisely we have the following equivalent as $\gamma\to 0$
\begin{multline} \label{exactequivalentinsertion}
Z_{\mu,\gamma,(z_i,\cX_i)_i}\sim e^{\gamma^{-2} \tf(\gL,(z_i,\cX_i)_i)}
\exp\left(-  2\pi\Lambda  \int_\U e^{ U(x)}\ln C(x,\U) \dd x \right)\\  \times \E\left[ \exp\left(-2\pi\Lambda\int_\U  e^{U(x)}  : X(x)^2 : \dd x   \right)  \right ]. 
\end{multline}

\item The field $\gamma X$ converges in probability in  $H^{-1}(\U)$   as $\gamma\to 0$  towards $U$.
\item Both  random measures $:e^{\gamma X}: \dd x$ and $e^{\gamma X} \dd x$ converge in law in the sense of weak convergence of measures towards $e^{U(x)}\,±\dd x$ as $\gamma\to 0$.
\item  the field $X-\gamma^{-1}U$ converges in law in  $H^{-1}(\U)$   towards a Massive Free Field in the metric $\hat{g}=e^{U(x)}\dd x^2$ with Dirichlet boundary condition and mass $8\pi^2\Lambda$ , that is a Gaussian field with covariance kernel given by the Green function of the operator $2\pi (8\pi^2\Lambda-\triangle_{\hat{g}})^{-1}$ and Dirichlet boundary condition.
\end{enumerate} 
\end{theorem} 

\subsection{The large deviation principle for LFT with insertions}

For $f\in H^1_0(\U)$, we consider the  weak solution $V$ of the perturbed Liouville equation (see Theorem \ref{th:liouville1})
\begin{equation}\label{bouchesource}
\triangle V=8\pi^2\Lambda e^{V(x)}-2\pi   f(x)- 2\pi \sum_i \cX_i\delta_{z_i}, \quad \text{with }V_{|\partial \U} =0,
\end{equation}
and we set  
$$\tf  (\Lambda,f)= - \frac{ 1}{4\pi  } \int_\U ( |\partial (V- H) |^2+16\pi^2\Lambda e^{ V(x)}) \,\dd x+ \int_\U(V-U-H)(x)f(x)\dd x,$$ 
 where  $U$ is the solution of the classical Liouville equation \eqref{eq:source}. The mapping $f \in H^1_0(\U)\mapsto \tf(\Lambda,f)-\tf(\gL)$ is still convex, G\^ateaux-differentiable and weakly lower semi-continuous.
 
 We define the Fenchel-Legendre transform $I^*$ of $\tf  (\Lambda,\cdot)-\tf  (\Lambda)$ as prescribed by \eqref{gignac}. We further define the set 
$$\mathcal{S}_{\text{source}}=\{h\in H^{-1}(\U); h=V-U-H;  V\text{ solution to \eqref{boucherie} for some function }f\in H^1_0(\U)\}.$$
The function $I^*$ is a good rate function, with $I^*(h)>0$ except for $h=0$.  For $h\in \mathcal{S}_{\text{source}}$, we have the following explicit expression 
\begin{equation*}
I^*(h)= E(V)-E(U)<+\infty, \quad  \text{if } h=V-U-H \text{ where } V \text{ solves } \eqref{bouchesource} \text{ for some }\in H^1_0(\U)  ,
\end{equation*} and 
$$\forall u\in H^1_0(\U)+H,\quad E(u)=  \frac{ 1}{4\pi  } \int_\U ( |\partial (u-H)(x)|^2+16\pi^2\Lambda e^{ u(x)})\dd x.$$
Finally, $\mathcal{S}_{\text{source}}+H$ is dense in $H^{-1}(\U)$ as it contains the set of $h$ such that $h,\triangle h\in H^1_0(\U)$.

\begin{theorem}\label{thldp:source} 
Assume that $\gamma\to 0,\mu\to\infty$ under the constraint \eqref{asymp}. 
Set $Y_\gamma=\gamma X-U$. The following LDP holds with good rate function $I^*$ on the space $H^{-1}(\U)$ equipped with the norm $|.|_{H^{-1}}$
$$-\inf_{h\in \mathring{A}}I^*(h)\leq  \gamma^2\liminf_{\gamma\to 0}\P_{\mu,\gamma,(z_i,\cX_i)_i}(Y_\gamma\in A)\leq  \gamma^2\limsup_{\gamma\to 0}\P_{\mu,\gamma,(z_i,\cX_i)_i}(Y_\gamma\in A)\leq -\inf_{h\in \overline{A}}I^*(h)$$ for each Borel subset $A$ of $H^{-1}(\U)$.
\end{theorem}

\subsection{Proofs} 
\noindent {\it Proof of Theorem \ref{th:source}.} We first compute the limit of the partition function $Z_{\mu,\gamma}$. 
\begin{align*}
Z_{\mu,\gamma,(z_i,\cX_i)_i}&=  \E\big[e^{- \frac{4\pi\Lambda}{\gamma^2} \int_\U e^{\gamma X(x)} e^{ \sum_i \cX_iG_\U(\cdot,z_i)}  \, \dd x }\big]    =\E\big[e^{- \frac{4\pi\Lambda}{\gamma^2} \int_\U :e^{\gamma X(x)}:e^{ \sum_i \cX_iG_\U(\cdot,z_i)}  C(x,\U)^{\frac{\gamma^2}{2}}\, \dd x }\big] .
\end{align*}
Let us set $V=U-H$ where $U$ is the solution of \eqref{eq:source}.
Note that $V$ satisfies
\begin{align}\label{platini}
V(x)= -4\pi \Lambda \int_\U e^{U(y)} G_\U(x,y)  \dd y.
\end{align}
Finally we set $Y=Y_\gamma=X-\gamma^{-1}V$.
 The computation as in \eqref{vampetta} yields
\begin{align*}
 &Z_{\mu,\gamma,(z_i,\cX_i)_i}=  e^{-  \frac{4\pi\Lambda}{\gamma^2} \int_\U e^{U(x)}(1-V(x))  \dd x} e^{ \frac{8\pi^2\Lambda^2}{ \gamma^2}   \iint_{\U^2}  e^{U(x)+U(y)}  G_\U(x,y)\dd x \dd y  }  \\
 & \quad  \times  \E\Big[ e^{- \frac{4\pi\Lambda}{\gamma} \int_\U e^{ U(x)} X(x) \dd x  - \frac{8\pi^2\Lambda^2}{ \gamma^2}   \iint_{\U^2}  e^{U(x)+U(y)} G_\U(x,y)\dd x \dd y  }  e^{- \frac{4\pi\Lambda}{\gamma^2} \int_\U e^{H(x)} e^{\gamma X(x)}-e^{U(x)} (1-\gamma Y(x)) \dd x }\Big] \\
 &\quad\quad\quad =  \exp\left( -  \frac{1}{4\pi\gamma^2} \int_\U(|\partial V(x)|^2+16\pi^2\Lambda e^{U(x)} )  \dd x \right) \E\Big[e^{
\int_\U e^{U(x)} (e^{\gamma X(x)}-1-\gamma X(x)) \dd x }\Big] .
\end{align*}
The last line is obtained by using \eqref{platini} to simplify the first term and by performing a Cameron-Martin transform in the expectation which by \eqref{platini} again has the property of shifting the field $X$ by an amount $\gamma^{-1}V$ and makes $Y$ centered.
The computation of the partition function as well as the other statements of Theorem \ref{th:source} are completed if one can show that Lemma \ref{waddle} also holds in the case when $U$ is the solution of \eqref{eq:source}.\qed

%
%
%
 
\vspace{2mm}
 
\begin{proof}[Proof of Lemma \ref{waddle} for $U(x)$ solution of \eqref{eq:source}]
What has to be done is to add a factor $e^{U(x)}$ in front of many terms and check that the proof still works.
We have to be a bit careful here because $e^{U(x)}$ is not bounded as it possesses singularities at the points where the mass is added.
However as these singularities are integrable this causes no major problem.
Let us mention a few modifications that are needed for the proof to work: note that before \eqref{llacer}, it is not true that 
$\int_\U e^{U(x)} (X-X_\gep)\dd x$ is of order $\gep$, but we still get a power of $\gep$ which is sufficient for our purpose. In the rest of the proof we just have to use that $e^{U(x)}$ is integrable.
\end{proof}

\noindent {\it Proof of Lemma \ref{lem:neymar} for $U(x)$ solution of \eqref{eq:source}.} 
What we need to do is to find a base of $\mathbb{L}^2(e^{U(x)}\,\dd x^2)$ which when suitably normalized is also a base of $H^1_0(\U)$.
Then the proof of Lemma \ref{lem:neymar} of the previous section applies.

\medskip

  First note that $e^{U}$ is in $\mathbb{L}^p(\U)$ for some $p>1$. Hence, one can consider the following Hilbert-Schmidt operator on the space $\mathbb{L}^2(e^{U(x)}\,\dd x )$
\begin{equation*} 
f \mapsto \int_\U G_\U(\cdot ,y) f(y) e^{U(y)}\,\dd y 
\end{equation*}
This symmetric operator can be diagonalized along an orthonormal (in $\mathbb{L}^2(e^{U(x)}\,\dd x)$) sequence $(e_j)_{j \geq 1}$ with associated eigenvalues $(\frac{1}{\lambda_j})_{j \geq 1}$ (decreasing order with repetition to account for multiple eigenvalues). We stress that $\sum_j \lambda_j^{-2}<+\infty$. Therefore we have
\begin{equation}\label{dribbledebenarfa}
\frac{e_j(x)}{\lambda_j}= \int_\U G_\U(x,y) e_j(y) e^{U(y)}\,\dd y .
\end{equation} 
 By using Cauchy-Schwarz, we get
\begin{equation*}
| \frac{e_j(x)}{\lambda_j} |  \leq \Big(\int_\U G_\U(x,y)^2 e^{U(y)}\,\dd y\Big)^{1/2}\Big(\int_\U e_j(y)^2 e^{U(y)}\,\dd y\Big)^{1/2}. 
\end{equation*}
Therefore \eqref{dribbledebenarfa} implies that $e_j$ is a continuous bounded function. One can then differentiate the expression \eqref{dribbledebenarfa} and see that the sequence $(e_j)_{j \geq 1}$ is in $H^1_0(\U)$; it is then standard to check that $(\frac{e_j}{\sqrt{\lambda_j}})_{j \geq 1}$ is an orthonormal sequence in $H^1_0(\U)$. In fact, the sequence $(\lambda_j)_j$ is the increasing sequence of eigenvalues of $-(2\pi)^{-1} \Delta_g$ with Dirichlet boundary conditions where $g$ is the metric tensor $e^{U(x)}\,\dd x^2$. It remains to show that the sequence $(\frac{e_j}{\sqrt{\lambda_j}})_{j \geq 1}$ is a \textbf{basis} of $H^1_0(\U)$. Consider a function $\varphi$ in $H^1_0(\U)$ which is orthogonal to every $(e_j)_{j\ge1}$ in $H^1_0$. Then as
\begin{equation}
\int_\U \varphi(x)  e_j(x)e^{U(x)}\dd x= - 2\pi (\gl_j)^{-1} \int_{\U} \varphi(x) \Delta e_j(x) \dd y=  
2\pi (\gl_j)^{-1} \int_{\U} \langle\partial \varphi(x), \partial e_j(x) \rangle\dd x=0.
\end{equation}
it is also orthogonal to all the $(e_j)_{j\ge 0}$ as an element of   $\mathbb{L}^2(e^{U(x)}\,\dd x)$ and thus is equal to zero. \qed 
%
%
%
%

\vspace{2mm}

 \begin{proof}[Proof of Theorem \ref{thldp:source}]
 
As in the proof of Theorem \ref{th:ldp}, the first task is to compute the Laplace transform of linear forms under the Liouville measure.

Let  
$V$ be the (deterministic) weak solution of \eqref{bouchesource} (see Corollary \ref{th:liouvillesource2}) and 
set 
\begin{equation}\label{deftheta}
T:= V-H.
\end{equation}
We also consider  
\begin{equation}\label{antonpolter}
\theta(x):= f(x)-4\pi \Lambda e^{V(x)}
\end{equation}
Note our definition together with \eqref{bouchesource} imply that  
\begin{equation}\label{guivarch}
T(x):=  \int_\U \theta(y) G_\U(x,y)\dd y
\end{equation}
We define $Y_\gamma$ to be a shifted version of the field $X$,  
\begin{equation}
Y_\gamma(x)=X-\frac{T}{\gamma}.
\end{equation}
We have  
\begin{align}
 & \E\big[\exp\left( \gamma^{-1}X(f)   - \frac{4\pi\Lambda}{\gamma^2} \int_\U e^{H(x)}e^{\gamma X(x)} \, \dd x \right) \big] \\
  &=\E\big[e^{\frac{1}{\gamma^2} \int_{\U}\gamma X(x)f(x)\dd x -\frac{4\pi\Lambda}{\gamma^2} \int_{\U} e^{V(x)}(1+\gamma Y_\gamma(x)) \, \dd x }e^{- \frac{4\pi\Lambda}{\gamma^2} \int_\U  e^{H(x)} e^{\gamma X(x)}-e^{V(x)}(1+\gamma Y_\gamma(x)) \, \dd x }\big]\nonumber  \\
 &= e^{   \frac{1}{\gamma^2} \int_\U (4\pi\Lambda T(x)e^{V(x)}-4\pi\Lambda e^{V(x)}) 
 \dd x+ \frac{1}{2 \gamma^2}   \iint_{\U^2}  \theta(x)\theta(y)  G_\U(x,y)\dd x \dd y }\nonumber\\
 &\times
 \E\big[ e^{\frac{1}{\gamma } \int_\U   X(x) \theta(x)\dd x- \frac{1}{2 \gamma^2}   \int_{\U^2}  \theta(x)\theta(y)  G_\U(x,y)\dd x \dd y} 
 e^{- \frac{4\pi\Lambda}{\gamma^2} \int_\U e^{H(x)} e^{\gamma X(x)}-e^{V(x)}(1+\gamma Y_\gamma(x)) \, \dd x }\big] .\label{valbuena}
\end{align}
 The usual Cameron-Martin tricks helps us to control the last term: the tilt shifts $X$ by an amount $T/\gamma$ (cf. \eqref{guivarch}) and has the effect of 
 centering $Y$. Similarly to \eqref{laudrup} and \eqref{toefting}, we have \textcolor{red}{(OK)}

\begin{equation}
 \iint_{\U^2}  \theta(x)\theta(y)G_\U(x,y)\dd x \dd y
= \frac{1}{2\pi}\int_{\U}  |\partial V(x)-\partial H(x)|^2 \dd x.
\end{equation}
and 
\begin{equation}
 \int_{\U} 4\pi \Lambda e^{V(x)} T(x) \dd x=-\frac{1}{2\pi}\int_{\U}  |\partial V(x)-\partial H(x)|^2 \dd x+  \int_{\U} f(x)(V(x)-H(x))\,\dd x,
\end{equation}
which yields  
\begin{equation}\begin{split}
 &  \E\big[\exp\left( \gamma^{-1} X(f)  - \frac{4\pi\Lambda}{\gamma^2} \int_\U e^{H(x)}e^{\gamma X(x)} \, \dd x \right) \big] \\
&\quad \quad =\exp\left(- \frac{ 1}{4\pi \gamma^2} \int_\U (  |\partial V(x)-\partial H(x)|^2+16\pi^2\Lambda e^{ V(x)}) \,\dd x+\frac{1}{\gamma^2}\int_\U f(x)(V(x)-H(x))\dd x\right)\\
&\quad \quad \quad \quad \times \E\Big[   \exp\left(- \frac{4\pi\Lambda}{\gamma^2} \int_\U e^{ V(x)} (e^{\gamma X(x)}-1- \gamma X(x) )\dd x \right)\Big].
\end{split}\end{equation}
and the last expectation converges towards a constant. This gives the exact asymptotic expression of the Laplace transform at exponential scale. 
\begin{align*}
\lim_{\gamma\to 0}\gamma^2\ln &\,\E_{\mu,\gamma,(z_i,\cX_i)_i}\Big[\exp\Big(\gamma^{-2}(\gamma X(f) -\int_\U U(x)f(x)\dd x)\Big)\Big]=\\
&- \frac{ 1}{4\pi  } \int_\U ( |\partial V-\partial H|^2+16\pi^2\Lambda e^{ V(x)}) \,\dd x+ \int_\U(V-U-H)(x)f(x)\dd x.
\end{align*}
We can then complete the proof by following the lines of Theorem \ref{th:ldp} (use Theorem \ref{th:liouvillesource} and Corollary \ref{th:liouvillesource2} to study the rate function).
 \end{proof}

  \appendix

\section{Solving the modified Liouville equation}\label{solving}
This section is devoted to solving the (eventually singular) Liouville equation as well as some variants. The techniques developed here are  known in the  differential geometry community and are close to \cite{battaglia}. Yet, we have not found references corresponding exactly to the results we need. Furthermore, the proofs are rather elementary and may help the reader (not necessarily familiar with these equations) to understand how it works.

\begin{theorem}\label{th:liouville1}
For  every  function $f$ belonging to $H^1_0(\U)$, the equation 
\begin{equation}\label{blanc}
\triangle U=8\pi^2 \Lambda e^{U }-2\pi f, \quad  U_{|\partial \U}=0
\end{equation} admits a  weak solution on $\U$  which is H\"older continuous $C^{1,\alpha}(\U)$ for all $\alpha<1$. 
\end{theorem}

\noindent {\it Proof.} Let us consider the   solution $g\in H^1_0(\U)$ of the equation $\triangle g=-2\pi f$ with boundary condition $g_{|\partial \U}=0$. Let us set $h(x)=8\pi^2 \Lambda e^{g(x)}$. It is then readily seen that $U$ is a weak solution to \eqref{blanc} if and only if $V=U-g$ is a weak solution to
\begin{equation}\label{djorkaeff}
\triangle V=h(x)e^{V(x)}, \quad  V_{|\partial \U}=0.
\end{equation}
 With the help of Sobolev-Orlicz space embeddings \cite{trudinger}, $H^1_0(\U)$ is continuously embedded into the Orlicz space with Young function $\Phi(t)=\exp(t^2)-1$. It results that $e^{g}\in \mathbb{L}^p(\U)$ for all $p>1$.

Let us consider the positive functional $E$ defined on $H^1_0(\U)$ 
$$E(V)=\int_\U\big(|\partial V(x)|^2+2 h(x)e^{V(x)}\big)\,dx: =\int_\U F(x,V(x),\partial V(x))\,dx. $$ 
Since $h\in \mathbb{L}^q(\U)$ for some $q>1$, the functional $E$ is indeed defined on $H^1_0(\U)$. Since $p \mapsto F(x,V,p)$ is convex and $F$ is greater or equal to $0$ the functional $E$ is weakly lower semi-continuous (see \cite[Theorem 1.6]{struwe}). Since $E(V)$ goes to infinity as $\int_\U |\partial V|^2 dx$ goes to infinity, $E(V)$ achieves its infimum in $H^1_0(\U)$ as a consequence of  \cite[Theorem 1.2]{struwe}. One can check that $\argmin E$ is reduced to one point ($E$ is strictly convex) which is a weak solution to \eqref{blanc}: we call it $V$. Once again,  with the help of Sobolev-Orlicz space embeddings, we know that $e^{V}\in \mathbb{L}^p(\U)$. By H\"older's inequality, the product $he^V=\Delta V$ belongs to $\mathbb{L}^p(\U)$ for all $p>1$.   Standard arguments of Sobolev embeddings allows us to conclude that $V$ is H\"older continuous on $\U$. \qed

\begin{proposition}\label{th:liouville2}
For  every  function $f,h$ belonging to $H^1_0(\U)$, we denote by $U_t$ the solution of the equation
\begin{equation}\label{cruyff}
\triangle U_t=8\pi^2 \Lambda e^{U_t }-2\pi (f+th), \quad  U_{t|\partial \U}=0.
\end{equation} 
Then the family $\big(\frac{U_t-U_0}{t}\big)_{t>0}$ strongly converges in $H^1_0(\U)$ towards the solution $V$ of the equation 
\begin{equation}\label{vanbasten}
\triangle V=8\pi^2 V\Lambda e^{U_0 }-2\pi h, \quad  V_{|\partial \U}=0.
\end{equation} 
\end{proposition}

\noindent {\it Proof.} First notice that \eqref{vanbasten} is linear in $V$ so that there are no troubles in establishing existence and uniqueness of a weak solution to this equation  (see e.g.\ \cite[Theorem 1.2]{struwe}). Furthermore, the Sobolev-Orlicz embedding entails that $\sup_{t\in]0,1]}\int_\U e^{2 U_t}\,dx<+\infty$ for all $p>1$ and hence (from \eqref{cruyff}) that $\Delta U$ is in $\mathbb{L}^2(\U)$. 
The standard Sobolev embedding then entails that  $M=\sup_{t\in ]0,1]}\sup_{x\in \U} |U_t(x)|<+\infty$. In what follows, we will consider a constant $D$ such that 
\begin{equation}\label{rijkaard}
|e^x-1-x|\leq Dx^2,\quad \text{ for all }|x|\leq 2 M.
\end{equation}

Set $V_t= \frac{U_t-U_0}{t}$. Furthermore, by considering the difference of \eqref{cruyff} evaluated at $t$ and $t=0$ and then integrating against a test function $\phi$ in $H^1_0(\U)$, we obtain
\begin{align}\label{gullit1}
\int_\U\langle \partial V_t,\partial\phi\rangle \,dx+8\pi^2 \Lambda \int_\U e^{U_0}t^{-1}(e^{U_t-U_0}-1)\phi\,dx=2\pi  \int_\U h \phi\,dx.
\end{align}
Taking $\phi=V_t$ and using the inequality  $x(e^{x}-1)\geq  0$, we deduce
 \begin{align}\label{gullit2}
\int_\U|\partial V_t|^2\,dx\leq 2\pi  \int_\U h V_t\,dx\leq C\Big(\int_\U|\partial V_t|^2\,dx\Big)^{1/2}\Big(\int_\U |h|^2\,dx\Big)^{1/2}.
\end{align}
We used the Poincar\'e inequality to get the last inequality. Hence the sequence $(V_t)_t$ is bounded in $H^1_0(\U)$, and has limit points when $t\to 0$ for the weak topology in $H^1_0(\U)$.
Let $\overline{V}$ be one of these limit points. By taking the limit along a subsequence converging to $\bar V$ in \eqref{gullit1} (and using \eqref{rijkaard} to get rid of the exponential term), we deduce that $\overline{V}$ is a weak solution to \eqref{vanbasten}. By uniqueness, $\overline{V}=V$ and is the weak limit of $(V_t)_t$. It remains to prove the convergence of the norms to get the strong convergence. By taking once again $\phi=V_t$ in \eqref{gullit1}, we get 
 \begin{align*} 
\lim_{t\to 0}\Big(\int_\U|\partial V_t|^2\,dx+8\pi^2 \Lambda \int_\U e^{U_0}t^{-1}(e^{U_t-U_0}-1)V_t\,dx\Big)=2\pi  \int_\U h V\,dx.
\end{align*}
The main difficult term is the integral containing the exponential term. With the help of \eqref{rijkaard}, we have
 \begin{align*} 
\int_\U e^{U_0}t^{-1}(e^{U_t-U_0}-1)V_t\,dx=\int_\U e^{U_0}|V_t|^2\,dx+H_t,\quad |H_t|\leq Dt\int_\U|V_t|^3\,dx.
\end{align*}
By the Rellich-Kondrachov theorem, the embedding $H^1_0(\U)\to \mathbb{L}^2(\U)$ is compact so that the first term in the right-hand side converges towards $\int_\U e^{U_0}|V|^2\,dx$. Furthermore as $V_t$ is bounded in $H^1_0(\U)$ , the Sobolev embedding entails that $\sup_{t\in]0,1]}\int_\U e^{U_0}|V_t|^3\,dx<+\infty$. Hence the second term goes to $0$. We deduce
 \begin{align*} 
\lim_{t\to 0} \int_\U|\partial V_t|^2\,dx=-8\pi^2 \Lambda \int_\U e^{U_0}V^2\,dx+2\pi  \int_\U h V\,dx=\int_\U|\partial V|^2\,dx.
\end{align*}
The proof is complete.\qed

\begin{proposition}\label{th:liouville3}
Assume that the family $(f_t)_{t>0}$ is weakly converging towards $f_0$ in $H^1_0(\U)$ as $t\to 0$. Denote by $U_t$ the solution of the equation
\begin{equation}\label{baresi}
\triangle U_t=8\pi^2 \Lambda e^{U_t }-2\pi f_t, \quad  U_{t|\partial \U}=0.
\end{equation} 
Then the family $ ( U_t )_{t>0}$ strongly converges in $H^1_0(\U)$ towards $U_0$.
\end{proposition}

\noindent {\it Proof.} The key points are first to observe that $(f_t)_{t>0}$ is strongly converging towards $f_0$ in $\mathbb{L}^2(\U)$ by using the Rellich-Kondrachov Theorem and  that $\sup_{t>0}\int_\U |f_t|^p\,dx<+\infty$ by the Sobolev embeddings. Then the arguments are quite similar to the proof of Proposition \ref{th:liouville2}: we can prove that $U_t-U_0$ converges strongly  to the a solution of \eqref{vanbasten} with $h=0$. Details are thus left to the reader.\qed

\begin{theorem}\label{th:liouvillesource}
Consider $z_1,\dots,z_p\in \U$ and $\cX_1,\dots\cX_p\in ]0,2[$. The equation 
\begin{equation}\label{juninho}
\triangle U=8\pi^2 \Lambda e^{U }-2\pi \sum_{i=1}^p\cX_i\delta_{z_i}, \quad  U_{|\partial \U}=0
\end{equation} admits a  solution on $\U$ such that $U-  \sum_{i=1}^p\cX_iG_\U(\cdot, z_i)$ is locally H\"older continuous on $\U$.
\end{theorem}

\noindent {\it Proof.} By using the same trick as in the proof of Theorem \ref{th:liouville1}, by setting $V=U-  \sum_{i=1}^p\cX_iG_\U(\cdot, z_i)$, it suffices to solve the equation
\begin{equation}\label{francescoli}
\triangle V(x)=h(x)e^{V(x)}, \quad  V_{|\partial \U}=0
\end{equation}
with $h(x)=8\pi^2 \Lambda e^{\sum_{i=1}^p\cX_iG_\U(\cdot, z_i)}$. Let us consider the functional
$$E(V)=\int_\U|\partial V(x)|^2+2h(x)e^{V(x)}\,dx$$ defined on $H^1_0(\U)$ (observe that $h\in L^q(\U)$ for some $q<q_\cX=2/\max_{i\in \{1,\dots,p\}} \cX_i$ where $q_\cX >1$). Therefore, one can use the same arguments than the ones in the proof of Theorem \ref{th:liouville1} to deduce from \cite[Theorem 1.2]{struwe} and \cite[Theorem 1.6]{struwe} that the functional $E$ is weakly lower semi-continuous and achieves its infimum in $H^1_0(\U)$. Moreover the reader can check that  $\argmin E$ is reduced to a point (by convexity) which is a solution of  \eqref{francescoli}. Let us call it $V$. With the help of Sobolev-Orlicz space embeddings \cite{trudinger}, $H^1_0(\U)$ is continuously embedded into the Orlicz space with Young function $\Phi(t)=\exp(t^2)-1$. It results that $e^{V}\in \mathbb{L}^p(\U)$ for all $p>1$. By H\"older's inequality, the product $he^V$ belongs to $\mathbb{L}^q(\U)$ for all $q<q_\cX$. Standard arguments of Sobolev embeddings allows us to conclude that $V$ is $\alpha$-H\"older continuous on $\U$ for all 
$\alpha<\max(1, 2(1-q_\cX^{-1}))$.\qed

\begin{proposition}\label{th:liouvillesource2}
For each function $f\in H^1_0(U)$ on $\U$, the equation 
\begin{equation}\label{thuram}
\triangle U=4\pi \Lambda e^{U }-2\pi \sum_{i=1}^p\cX_i\delta_{z_i}+2\pi f, \quad  U_{|\partial \U}=0
\end{equation} admits a solution on $\U$ such that $U-  \sum_{i=1}^p\cX_iG_\U(\cdot, z_i)$ is locally H\"older continuous on $\U$. 
\end{proposition}

\noindent {\it Proof.} It suffices to adapt the arguments of Theorem \ref{th:liouvillesource}.\qed
 
{\small
}

\end{document}